\documentclass[12pt]{article}
\usepackage{amsmath}
\usepackage{amssymb}
\usepackage{amsthm}
\usepackage{graphicx}
\usepackage{mathtools}
\usepackage{tikz}

\newcommand{\HH}{\mathcal{H}}
\newcommand{\LL}{\mathcal{L}}
\newcommand{\IP}{\mathbb{P}}
\newcommand{\IE}{\mathbb{E}}

\newcommand{\eps}{\varepsilon}

\newcommand{\1}[1]{{\mathbf 1}_{#1}}

\newcommand{\Z}{{\mathbb Z}}
\newcommand{\N}{{\mathbb N}}
\newcommand{\R}{{\mathbb R}}

\newcommand{\TT}{{\mathcal T}}

\newcommand{\G}{{\mathcal G}}
\newcommand{\s}{{\mathcal S}}
\newcommand{\sss}{{\mathbf s}}
\newcommand{\D}{{\mathcal D}}
\newcommand{\U}{{\mathcal U}}
\newcommand{\J}{{\mathcal J}}
\newcommand{\AAA}{{\mathfrak A}}

\newcommand{\QQ}{{\mathcal Q}}

\newcommand{\M}{{\mathcal M}}
\newcommand{\A}{{\mathcal A}}

\newcommand{\PP}{{\mathcal P}}
\newcommand{\FF}{{\mathcal F}}

\newcommand{\MM}{{\mathfrak M}}
\newcommand{\dd}{{\mathfrak d}}
\newcommand{\CC}{{\mathfrak C}}

\newcommand{\hP}{{\widehat P}}
\newcommand{\tP}{{\widetilde P}}

\newcommand{\din}{{\mathop{\text{deg}_{\text{in}}}}}
\newcommand{\dout}{{\mathop{\text{deg}_{\text{out}}}}}

\newcommand{\card}{{\mathop{\text{card}}}}

\let\phi=\varphi

\newtheorem{theo}{Theorem}[section]
\newtheorem{lem}[theo]{Lemma}
\newtheorem{df}[theo]{Definition}
\newtheorem{prop}[theo]{Proposition}

\newtheorem{conj}[theo]{Conjecture}

\title{Equilibria in the Tangle}

\author{Serguei Popov$^{1,}$\thanks{corresponding author
} 
\and Olivia Saa$^{2}$ \and Paulo Finardi$^{3}$}

\begin{document}
\maketitle

{\footnotesize 
\noindent $^{~1}$Department of Statistics, Institute of Mathematics,
 Statistics and Scientific Computation, University of Campinas --
UNICAMP, rua S\'ergio Buarque de Holanda 651,
13083--859, Campinas SP, Brazil\\
\noindent e-mails: \texttt{popov@ime.unicamp.br}
 and \texttt{serguei.popov@iota.org}

\noindent $^{~2}$Department of Applied Mathematics,
 Institute of Mathematics and  Statistics,
 University of S\~ao Paulo --
USP, rua do Mat\~ao 1010, 05508--090,
 S\~ao Paulo SP, Brazil\\
\noindent e-mail: \texttt{olivia@ime.usp.br}
 and \texttt{olivia.saa@iota.org}

\noindent $^{~3}$Institute of Computing -- IC,
University of Campinas -- UNICAMP, 
av.\ Albert Einstein 1251,
13083--852, Campinas SP, Brazil\\
\noindent e-mail: \texttt{ra144809@ic.unicamp.br}

}

\begin{abstract}
 We analyse
the Tangle --- a DAG-valued stochastic process
where new vertices get attached to the graph at Poissonian
times, and the attachment's locations are chosen by means of 
random walks on that graph. 
These new vertices, also thought of as ``transactions'',
are issued by many players (which are the nodes of the network),
independently.
The main application of this model is that it is used as 
a base for the IOTA cryptocurrency system~\cite{IOTA}.
We prove existence
of ``almost symmetric'' Nash equilibria for the system where 
a part of players tries to optimize their attachment strategies.
Then, we also present simulations that show that the ``selfish''
players will nevertheless cooperate with the network by choosing
attachment strategies that are similar 
to the ``recommended'' one.
 \\[.3cm]\textbf{Keywords:} random walk, Nash equilibrium,
  directed acyclic graph, cryptocurrency, tip selection, IOTA
\\[.3cm]\textbf{AMS 2010 subject classifications:}
Primary 91A15. Secondary 60J20, 68M14.
\end{abstract}

\section{Introduction}
\label{s_intro}
In this paper we study \emph{the Tangle},
a stochastic process on the space of (rooted) Directed
Acyclic Graphs (DAGs). This process ``grows'' in time, in the 
sense that new vertices are attached to the graph 
according to a Poissonian clock, but no vertices/edges
are ever deleted.
When that clock rings, 
a new vertex appears and attaches itself to 
locations that are chosen with the help of 
certain random walks on the state of the process in the
\emph{recent past} (this is to model the network propagation delays);
these random walks therefore play the key role in the model. 

Random walks on random graphs can be thought of 
as a particular case of Random Walks in Random Environments:
here, the transition probabilities are functions
of the graph only, i.e., there are no additional 
variables, such as conductances\footnote{this refers
to the well-known relation between reversible Markov
chains and electric networks, 
see e.g.\ the classical book~\cite{DS84}}
 etc., attached to 
the vertices and/or edges of the graph.
Still, this subject is very broad, and one can find
many related works in the literature.
One can mention the internal DLA models
(e.g.\ \cite{JLS14} and references therein), random walks 
on Erd\"os-R\'enyi graphs~\cite{CFP17,JLS14},
or random walks on the preferential attachment graphs~\cite{CF07},
which most closely resembles the model of this paper.

The motivation for studying the particular model 
presented in this paper stems from the fact that it is applied
 in the IOTA cryptocurrency~\cite{IOTA,Tangle}.
 The IOTA is an ambitious project started in 2015, it 
aims to provide a globally scalable system capable 
of processing payments and storing data.
One of its distinguishing features is that it
 uses (nontrivial) DAGs as the primary ledger
for the transactions' data\footnote{we also cite~\cite{hashgraph,byteball,dagcoin,spectre}
which deal with other approaches
to using DAGs as distributed ledgers}. 
This is different from ``traditional''
cryptocurrencies such as the Bitcoin, where that data is stored
in a sequence of blocks\footnote{that is, the underlying
graph is essentially~$\Z_+$ 
(after discarding finite forks)}, 
also known as \emph{blockchain}. An important observation,
which motivates the use of more general DAGs instead of blockchains
is that the latter \emph{scale} poorly.
Indeed, it is not hard to see that the chain of blocks of finite size, 
which can only be produced at regular discrete time intervals, produces 
a throughput bottleneck and leads to high transaction fees that need 
to be paid to the miners (which is by design).
Also, when the network
is large, it is difficult for it to achieve consensus
on which blocks are ``valid'' in the situations when
 the new blocks come too frequently.
 If one wants to remove
the fees and allow the system to scale, the natural idea would thus 
be to eliminate the bottleneck and the miners.

This is, of course, easier said than done --- it raises all sorts of new questions. Where should the next block/transaction/vertex be attached? 
Who will vet the transactions for consistency and why? 
How can it be secure against possible attacks? 
How will consensus be achieved? 
These questions do not have trivial answers. 
The paper~\cite{Tangle} presented an
 idea for an architecture which could potentially 
resolve these issues. In that system, each transaction, 
represented by a vertex in the graph, would approve two 
previous transactions it selects using a particular class of random walks. 
To eliminate the 
transaction fees, it was necessary
 to first eliminate the miners --- after all, 
if one wants to design a feeless system, there cannot be a dichotomy
of ``miners'' who serve the ``simple users''. 
This bifurcation of roles
between the miners and transactors naturally leads to transaction 
fees because the miners have some kind of resource that others do not
have and they will use this monopoly power to extract rents,
in the form of transaction fees, or block rewards, or both. 
Therefore, to eliminate the fees, all the users would 
have to fend for themselves. The 
main principle of such a system
would be ``help others, and others will help you''.

You can help others by approving their transactions; others can help 
you by approving your transactions. Let us call ``tips''
 the transactions 
which do not yet have any approvals; 
all new-coming transactions are tips 
at first. The idea is that, by approving a transaction, you also 
indirectly approve all its ``predecessors''. 
It is intuitively clear that, 
to help the system progress, 
the incoming transactions must approve tips 
because this adds new information to the system. However, due to the 
network delays, it is not practical to 
\emph{impose} that this must happen 
--- how can one be sure that what one believes 
to be a tip has not already been 
approved by someone else maybe 0.1 seconds ago?

In any case, if everybody collaborates with everybody 
--- only approving recent and ``good'' 
(non-contradicting) transactions, then 
we are in a good shape. On the other hand, for someone who only cares
about themself, a natural strategy would just be to choose a couple 
of old transactions and approve them all the time without having to do 
the more cumbersome work of checking new transactions for consistency
thereby adding new information to the system. If everybody behaves in
this way, then no new transactions will be approved, and the network
will effectively come to a halt. Thus, if we want it to work, 
we need to incentivize the participants to collaborate and 
approve each other’s recent transactions.
Therefore, in some sense,
it is all about the incentives. 
Everybody wants to be helped by others, but, not everybody cares 
about helping others themselves. To resolve this without having 
to introduce monetary rewards, we could instead think of a reward 
as simply not being punished by others. 
So, we need to slightly amend
the above main principle --- it now reads: 
``Help others, and the others 
will help you; however, if you choose not to help others, others 
will not help you either''. 
When a new transaction references two 
previous transactions, it is a statement of ``I vouch for these 
transactions which have not been vouched for before, as well as all 
their predecessors, and their success is tied to my success''.
It was suggested in~\cite{Tangle} that the
Markov Chain Monte Carlo (MCMC) tip selection algorithm 
(more precisely, 
the family of tip selection algorithms)
would have these properties.

This paper mainly deals with the following question:
what if some participants of the network are trying to minimize
their \emph{costs} by adopting a behavior different from the 
``default'' one?
How will the system behave in such circumstances?
In other words, are there enough incentives for the participants
to ``behave well''?
To address these kinds of questions, we first provide general
arguments to prove existence
of ``almost symmetric'' Nash equilibria for the system,
see Section~\ref{s_Nash_theory}.
Although one can hardly access the explicit form of
these equilibria in a purely analytical way,
 simulations presented in Section~\ref{s_simulations}
 show that the ``selfish''
players will typically still choose
attachment strategies that are similar to the default one,
meaning that they would prefer \emph{cooperating} with the network 
rather than simply \emph{using} it).

Let us stress also that, in this paper, we consider 
only ``selfish'' players, i.e., those who only care about their own 
costs but still want to use the network 
in a legitimate way\footnote{i.e., want to issue valid
transactions and have them confirmed by the rest of the network}.
We do not consider the case when there are ``malicious'' 
ones, i.e., those who want to disrupt the network even
 at a cost to themselves.
We are going to treat several types of attacks against the 
network in the subsequent papers. 

This paper is organized in the following way.
In Section~\ref{s_description} we first introduce some
notations and define the objects we are working with; 
then, in Section~\ref{s_attachment} we describe the ``recommended''
algorithm of how the nodes choose where to attach 
a new transaction, and then discuss some basic properties
of it, also formulating an open problem about the asymptotic
behavior of the total number of tips. Section~\ref{s_Nash_theory}
contains the main ``theoretical'' advances of this paper. 
There, we first discuss what is a ``strategy'' that could 
be used by a selfish player, and then (Section~\ref{s_restr_str})
make some further assumptions necessary to formulate our 
main results 
(which are placed in Section~\ref{s_mainres}). Then, we prove these
results in Section~\ref{s_proofs}.
Section~\ref{s_simulations} discusses some simulation
results, mainly in the case where the selfish players try to use 
a very natural ``greedy'' attachment strategy
 (Section~\ref{s_1dim_Nash}).
In Section~\ref{s_conclusions} one will find
conclusions and some final remarks.

\section{Description of the model}
\label{s_description}
In the following we introduce 
the mathematical model 
describing the Tangle~\cite{Tangle}.

Let $\card(A)$ stand for the cardinality of (multi)set~$A$. 
Consider an oriented multigraph $\TT=(V,E)$, where~$V$ is the set 
of vertices\footnote{one can think of vertices
as transactions} and~$E$ is the multiset of edges.
For $u,v\in V$, we say that~$u$
\emph{approves}~$v$, if $(u,v)\in E$. 
 For a vertex $v\in V$, let us denote by 
\begin{align*}
 \din(v) &= \card\{e=(u_1,u_2)\in E : u_2=v\}, \\
 \dout(v) &= \card\{e=(u_1,u_2)\in E : u_1=v\}
\end{align*}
the ``incoming'' and ``outgoing'' degrees of the vertex~$v$
(counting the multiple edges).
In the following, we refer to multigraphs simply as graphs.
We use the notation~$\A(u)$ for the set of the vertices approved by~$u$.
We say that~$u\in V$ \emph{references}~$v\in V$
if there is a sequence of sites $u=x_0,x_1,\ldots,x_k=v$
such that $x_{j}\in\A(x_{j-1})$ for all $j=1,\ldots, k$,
i.e., there is a directed path from~$u$ to~$v$.
If $\din(w)=0$ (i.e., there are no edges pointing to~$w$), then
we say that~$w\in V$ is a \emph{tip}.

Let~$\G$ be the set of all directed acyclic graphs
(also known as DAGs, that is, oriented graphs 
without cycles)  $G=(V,E)$ with the following
properties (see Figure~\ref{f_DAG}).
\begin{figure}
 \centering \includegraphics{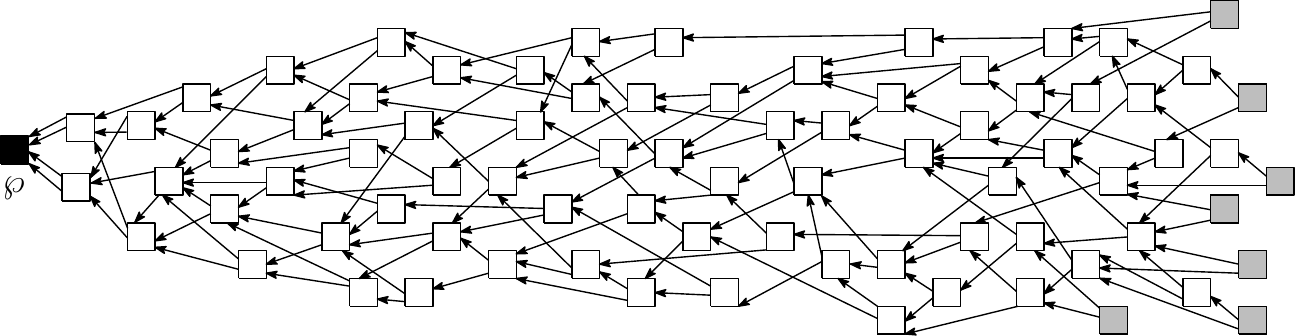} 
\caption{On the DAGs we are considering: the genesis vertex is on the 
left, and the tips are grey}
\label{f_DAG}
\end{figure}
\begin{itemize}
 \item The graph~$G$ is finite and 
the multiplicity of any edge is at most two
(i.e., there are at most two edges linking the 
 same vertices).
 \item There is a distinguished vertex $\wp\in V$ such that
$\dout(v)=2$ for all $v\in V\setminus\{\wp\}$,
and $\dout(\wp)=0$. 
This vertex $\wp$ is called \emph{the genesis}.
 \item Any $v\in V$ such that $v\neq \wp$ references~$\wp$;
that is, there is an oriented path\footnote{not necessarily
unique} from~$v$ to $\wp$.
\end{itemize}
We now describe the tangle as a continuous-time Markov process
on the space~$\G$.
The state of the tangle at time~$t\geq 0$ is a DAG
$\TT(t)=(V_{\TT}(t),E_{\TT}(t))$,
where $V_{\TT}(t)$ is the set of vertices and $E_{\TT}(t)$
is the multiset of directed edges at time~$t$.
The process's dynamics are described in the following way:
\begin{itemize}
 \item The initial state of the process is defined by
$V_{\TT}(0)=\wp$, $E_{\TT}(0)=\emptyset$.
 \item The tangle \emph{grows with time}, that is, 
$V_{\TT}(t_1)\subset V_{\TT}(t_2)$
and $E_{\TT}(t_1)\subset E_{\TT}(t_2)$ whenever $0\leq t_1<t_2$.
 \item For a fixed parameter $\lambda>0$, there is 
a Poisson process of incoming \emph{transactions};
these transactions then become the vertices of the tangle.
\item Each incoming transaction
chooses\footnote{the precise
selection mechanism
will be described below} 
two vertices~$v'$ and~$v''$ (which, in general, may coincide), 
and we add the edges $(v, v')$ and $(v,v'')$.
We say in this case that this new transaction
was \emph{attached} to~$v'$ and $v''$ 
(equivalently, $v$ \emph{approves} $v'$ and~$v''$).
\item Specifically, if a new transaction~$v$ arrived at time~$t'$,
then $V_{\TT}(t'+)=V_{\TT}(t')\cup \{v\}$, 
and $E_{\TT}(t'+)=E_{\TT}(t)\cup\{(v,v'),(v,v'')\}$.
\end{itemize}
 
Let us write
\begin{align*}
\PP^{(t)}(x) &= \big\{y\in \TT(t) : y \text{ is referenced by }x\big\},\\
 \FF^{(t)}(x) &= \big\{z\in \TT(t) : z \text{ references }x\big\}
\end{align*}
for the ``past'' and the ``future'' with respect to~$x$ (at time~$t$).
Note that these introduce
a \emph{partial
order} structure on the tangle. Observe that, if~$t_0$ is the 
time moment when~$x$ was attached to the tangle, then
$\PP^{(t)}(x)=\PP^{(t_0)}(x)$ for all~$t\geq t_0$.
We also define the \emph{cumulative weight} 
 $\HH^{(t)}_x$ of the vertex~$x$ at time~$t$ by
\begin{equation}
\label{df_cum_weight}
 \HH^{(t)}_x = 1 + \card\big(\FF^{(t)}(x)\big);
\end{equation}
that is, the cumulative weight of~$x$ is one\footnote{its ``own weight''}
plus the number of vertices that reference it.
Observe that, for any $t>0$, if $y$ approves~$x$ then
$\HH^{(t)}_x-\HH^{(t)}_y\geq 1$, and the inequality
is strict if and only if there are vertices different from~$y$
which also approve~$x$.
Also note that the cumulative weight of any tip is equal to~$1$.

There is some data 
 associated to each vertex (transaction), 
created at the moment when that transaction was attached 
to the tangle. 
 The precise nature of that data is not relevant
for the purposes of this paper, so we assume that it is
an element of some (unspecified, but finite) set~$\D$; 
what is important, however,
is that there is a natural way to say if the set of 
vertices is \emph{consistent} with respect to the data they 
contain\footnote{one may think that the data refers to value
transactions between accounts, and consistency means that 
no account has negative balance as a result,
and/or the total balance has not increased}.
When it is necessary to emphasize that the vertices
of $G\in \G$ contain some data, we consider
the \emph{marked} DAG  $G^{[\dd]}$
to be  $(G,\dd)=(V,E,\dd)$, where~$\dd$ is a function~$V\to\D$.
We define $\G^{[\dd]}$ to be the set of all marked DAGs
$(G,\dd)$, where $G\in\G$.

A note on terminology: we reserve the term ``node'' for 
entities that participate in the system by issuing transactions
(which are, by their turn, \emph{vertices} of the tangle graph).
That is, the ``players'' mentioned in Section~\ref{s_intro} 
are nodes.

\subsection{On attaching a new transaction to the Tangle}
\label{s_attachment}
There is one very important detail that has not been explained,
namely: how does a newly arrived transaction choose which two 
vertices in the tangle it will approve, i.e., what 
is the \emph{attachment strategy}? 
Notice that, in principle, it would be good\footnote{good in the 
sense described in Section~\ref{s_intro}}
for the whole system if the new transactions
always prefer to select tips as attachment places, since this 
way more transactions would be ``confirmed''\footnote{we 
discuss the exact meaning of this later; for now, think that 
``confirmed'' means ``referenced by many other transactions''}.
In any case, it is quite clear that the appropriate 
choice of the attachment strategy is essential for 
the correct functioning 
of the system, whatever this could mean.

It is also important to comment that the attachment
strategy of a network node is something ``internal''
to it; what others can see, are the \emph{attachment choices}
of the node, but the mechanism behind them need not be 
publicly known. For this reason, an attachment strategy
cannot be \emph{imposed} in the protocol. 

We now describe a possible choice of the 
attachment strategy, used to determine where the incoming 
transaction will be attached.
It is also known as the 
\emph{recommended tip selection algorithm} \cite{Tangle},
since, due to reasons described above, the recommended 
nodes' behavior is always to try to approve tips.
We stress again, however, that approving only tips
is not imposed in the protocol, since
there is usually no way to know if a node ``knew''
if the transaction it approved was already approved
by someone else before (also, there is no way 
to know which approving transaction was the first).

Let us denote by~$\LL(t)$ the set of all vertices that
are tips at time~$t$, and let $L(t)=\card(\LL(t))$.
To model the network propagation delays, we introduce a 
parameter~$h>0$, and assume that at time~$t$ only $\TT(t-h)$
is known to the entity that issued the incoming transaction.
We then define the \emph{tip-selecting random walk},
in the following way. It depends on a parameter~$q$ 
(the backtracking probability) and on a function~$f$.
 The initial state of the random walk
is the genesis~$\wp$\footnote{although in practical implementations
one may start it in some place closer to the tips},
and it is stopped upon hitting the set~$\LL(t-h)$.
It is important to observe that $v\in\LL(t-h)$
does not necessarily mean that~$v$ is still a tip at time~$t$. 
Let $f:\R_+\to\R_+$ be a monotone non-increasing function.
The transition probabilities
 of the walkers are defined in the following
 way: the walk \emph{backtracks} (i.e., jumps to a randomly chosen
site it approves) with probability~$q\in [0,1/2)$;
 if~$y$ approves~$x\neq \wp$,
then the transition probability~$P^{(f)}_{xy}$
 is proportional to $f(\HH_x-\HH_y)$,
that is,
 \begin{equation}
  \label{trans_probs}
   P^{(f)}_{xy} = 
\begin{cases}
\displaystyle
\frac{q}{\card(\A(x))}, & \text{ if } y\in\A(x),\\
\displaystyle\frac{(1-q)f(\HH^{(t-h)}_x-\HH^{(t-h)}_y)}
{\sum_{z: x \in\A(z)}
 f(\HH^{(t-h)}_x-\HH^{(t-h)}_z)}, 
& \text{ if } x\in\A(y),\\
0, & \text{ otherwise;}
\end{cases}
\end{equation}
for $x=\wp$ we define the transition probabilities as above,
but with $q=0$.
In words, the walker \emph{backtracks} (i.e., moves one step 
away from the tips) with (total) probability~$q$, and 
advances one step towards the tips with (total) probability~$(1-q)$
and relative weights as above.
Note that the fact that $q<1/2$ guarantees that the 
random walk eventually reaches a tip\footnote{more precisely,
reaches a vertex that the node \emph{assumes} 
to be a tip} almost surely. 
In what follows, we will mostly assume that 
$f(s)=\exp(-\alpha s)$ for some $\alpha\geq 0$.
We use the notation $P^{(\alpha)}$ for the transition
probabilities in this case. 
Intuitively, the smaller is the value of~$\alpha$, the more
\emph{random} the walk is\footnote{physicists would 
call the case of small~$\alpha$
\emph{high temperature regime}, and the case of large~$\alpha$
\emph{low temperature regime} 
(that is, $\alpha$ stands for the inverse temperature)}.
It is worth observing that the case $q=0$ and $\alpha\to \infty$
corresponds to the GHOST protocol of~\cite{SZ}
(more precisely, to the obvious
generalization of the GHOST protocol to the case when a 
tree is substituted by a DAG).

Now, to select two tips~$w_1$ and~$w_2$ where our transaction will be
attached, just run two independent random walks as above,
and stop when on the first hit~$\LL(t-h)$.
One can also require that~$w_1$ should be different
 from~$w_2$; for that, one may re-run the second random walk
in the case its exit point happened to be the same
as that of the first random walk.
Observe that $(\TT(t),t\geq 0)$ is a continuous-time
transient Markov process on~$\G$; since the state space is quite large,
it is difficult to analyse this process. In particular, for a fixed time~$t$,
it is not easy to study the above random walk since it takes
place on a \emph{random} graph, e.g., can be viewed as a
random walk in a random environment;
it is common knowledge that random walks in random environments
are notoriously hard to deal with.

Some motivation for choosing the attachment strategy in the above
way is provided in~\cite{Tangle}. Very briefly, it encourages
the nodes to choose \emph{recent} transactions for approval
 (since a transaction which approved a couple of old transactions,
also known as \emph{lazy tip}, is unlikely to be chosen by the 
above random walk,
due to the large difference in cumulative weights in the argument
of~$f$ in~\eqref{trans_probs})
and also gives protection against certain kinds of attacks
(e.g., the double-spending attack).

Let $\gamma_0\in (0,1)$ be some number, typically close to~$1$.
We say that a transaction is \emph{confirmed with confidence~$\gamma_0$}
 if, with probability at least~$\gamma_0$,
 the large-$\alpha$ random walk\footnote{recall that the 
large-$\alpha$ random walk is ``more deterministic''} 
ends in a tip which references that transaction.
It may happen that a 
transaction does not get confirmed (even, possible, does not
get approved a single time), 
and becomes orphaned forever.
Let us define the event
\[
 \U = \{\text{every transaction eventually gets approved}\}.
\]

We believe that the following statement holds true; 
however, we have only a
heuristical argument in its 
favor, not a rigorous proof. 
In any case, it is mostly of theoretical interest, since, as explained
below, in practice we will find ourselves in the situation where $\IP[\U]=0$.
We therefore state it as
\begin{conj}
\label{conj_integral_criterion}
 It holds that 
\begin{equation}
\label{eq_integral_criterion} 
 \IP[\U] = 
 \begin{cases}
  0, & \displaystyle\text{if }\int_0^{+\infty} f(s)\, ds  < \infty,
\vphantom{\int\limits_{A_A}}\\
  1,\vphantom{\int\limits^{B^B}} 
& \text{if }\displaystyle\int_0^{+\infty} f(s)\, ds  = \infty.
 \end{cases}
\end{equation}
\end{conj}
\begin{proof}[Explanation.]
First of all, it should be true that $\IP[\U]\in\{0,1\}$
since~$\U$ is a \emph{tail event} with respect to the natural
filtration; however, it does not seem to be very easy
to prove the 0--1 law in this context --
recall that we are dealing with a transient Markov process
on an infinite state space.
 Next, consider a tip~$v_0$ which got attached to the tangle
at time~$t_0$, and assume that it is still a tip 
at time $t\gg t_0$; also, assume that, among all tips, $v_0$
is ``closest'', in some suitable sense, to the genesis.
\begin{figure}
 \centering \includegraphics[width=\textwidth]{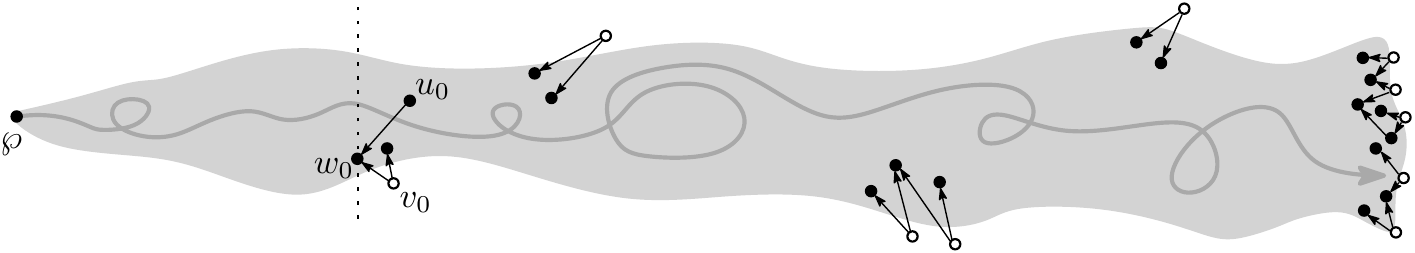} 
\caption{The walk on the tangle and tip selection. Tips are circles,
and transactions which were approved at least once are disks.}
\label{f_tip_walk}
\end{figure}
Let us now think of the following question: what is the 
probability that~$v_0$ will still be a tip at time~$t+1$?

Look at Figure~\ref{f_tip_walk}: during the time interval $[t,t+1)$,
$O(1)$ new particles will arrive, and the corresponding
walks will travel from the genesis~$\wp$ looking for tips.
Each of these walks will have to cross 
the dotted vertical segment on the picture,
and
with positive probability at least 
one of them will pass through~$w_0$, 
one of the vertices approved by~$v_0$.
Assume that~$w_0$ was already confirmed,
i.e., it is connected to the right end of the tangle
via some other transaction~$u_0$ that approves~$w_0$.
Then,  it is clear (but not easy to prove!) that 
the cumulative weight of both~$u_0$ and~$w_0$ should be~$O(t)$,
and so, when in~$w_0$, the walk will jump to the tip~$v_0$
with probability $f(O(t))$.

This suggests that the probability
that~$v_0\in\LL(t+1)$ (i.e., that~$v_0$ still is tip at time~$t+1$)
is $f(O(t))$,
and the Borel-Cantelli lemma\footnote{to be precise, 
a bit more refined argument is needed since 
the corresponding events are not independent} 
gives that the probability 
that~$v_0$ will be eventually approved is less than~$1$ 
or equal to~$1$
depending on whether $\sum_n f(n)$ converges 
or diverges; the convergence (divergence) of the sum
 is equivalent to convergence (divergence) of the integral
in~\eqref{eq_integral_criterion} due to the monotonicity of
the function~$f$. A standard probabilistic argument\footnote{which
is also not so easy to formalize in these circumstances} 
would then imply that if the probability that \emph{a given} tip
remains orphaned forever is uniformly positive, then the probability
that \emph{at least one} tip
remains orphaned forever is equal to~$1$.
\end{proof}

One may naturally think that it would be better to choose
the function~$f$ in such a way that, almost surely, every
tip eventually gets confirmed.
However,
as explained in Section~4.1 of~\cite{Tangle}, there is 
a good reason to choose a rapidly decreasing function~$f$,
because this defends the system against nodes' misbehavior 
and attacks. The idea is then to assume that a transaction
which did not get confirmed during a sufficiently long period
of time is ``unlucky'', and needs to be reattached\footnote{in fact,
the nodes of the network may adopt a rule that instructs
to delete the transactions that are older than~$K$ and 
still are tips from their databases} to the tangle.
Let us fix some~$K>0$: it stands for the time 
when an unlucky transaction
is reissued (because there is already very little hope
that it would be confirmed ``naturally'').
We call a transaction issued less than~$K$ time units ago
``unconfirmed'', and if a transaction was issued more
than~$K$ time units ago and was not confirmed, we call 
it ``orphaned''.
In the following, we assume that the system is \emph{stable},
in the sense that the ``recent'' unconfirmed transactions
do not accumulate and the time until a transaction is confirmed
 does not depend a lot on the moment when it appeared 
in the system\footnote{simulations 
indicate that this is indeed
the case when~$\alpha$ is small (cf.\ a recent 
paper~\cite{KG18}); however, it 
 is not guaranteed to happen for large values of~$\alpha$}.
We prefer not to elaborate on the exact mathematical definition
of stability here, since it requires considering a certain
compactification of the space of DAGs (which essentially 
amounts to considering DAGs with ``genesis at minus infinity''),  
but, hopefully, the idea is intuitively clear anyway.

In that stable regime, let~$p$ be the probability that a transaction
is confirmed~$K$ time units after it was issued
for the first time;
 the number of times a transaction should be issued
to achieve confirmation is then a Geometric
random variable with parameter~$p$ (and, therefore, with
expected value~$p^{-1}$); so, the mean time until the 
transaction is confirmed is $K/p$.
Let us then recall the
following remarkable fact belonging to the queuing theory, 
known as the Little's formula 
(sometimes also referred to as the Little's theorem
or the Little's identity): 
\begin{prop}
\label{prop_Little}
Suppose that $\lambda_a$ is
the arrival rate, $\mu$ is the mean number of customers in the system,
and~$T$ is the mean time a customer spends in the system.
Then $T=\mu/\lambda_a$.
\end{prop}

\begin{proof} 
 See e.g.\ Section~5.2 of~\cite{C81}. To understand
intuitively why this fact holds true,
 one may reason in the following way: 
assume that, while in the system, each customer pays money 
to the system with rate~$1$. Then,
at large time~$t$, the total amount of money earned 
by the system would be (approximately)~$\mu t$ on one hand,
and~$T\lambda_a t$ on the other hand. 
Dividing by~$t$ and then sending~$t$ to infinity,
we obtain $\mu = T\lambda_a$.
\end{proof}

Little's formula then implies\footnote{in the language
of queuing systems, a reissued transaction is a customer
which goes back to the server after an unsuccessful 
service attempt} the following
(recall that~$\lambda$ is the rate of the incoming transactions flow,
not counting reattachments)
\begin{prop}
\label{prop_mean_unconf}
The average number
of unconfirmed transactions\footnote{we regard all reattachments as a single
trasaction, and if one of the reattachments is confirmed, the transaction
is considered confirmed}
in the system is equal to $p^{-1}\lambda K$.
\end{prop}

\begin{proof} 
Indeed, apply Proposition~\ref{prop_Little} with
$\lambda_a=\lambda$ (think of a transaction which was 
reattached as a customer which returns to the server 
after an insuccessful service attempt; this way, the 
incoming flow of customers still has rate~$\lambda$).
As observed before, the mean time spent by a customer
in the system is equal to~$K/p$.
\end{proof}

When the tangle contains data, this, in principle, can 
make transactions incompatible between each other.
 In this case one may choose
more sophisticated methods of tip selection.
As we already mentioned\footnote{recall the discussion around
 $f(s)=\exp(-\alpha s)$ right after~\eqref{trans_probs}}, 
selecting tips with larger values
of~$\alpha$ provides better defense against attacks and misbehavior;
however, smaller values of~$\alpha$ make the system more stable
with respect to the transactions' confirmation times.
An example of ``mixed-$\alpha$'' strategy is the following. 
Define the ``model tip'' $w_0$ as a result of the random 
walk with large~$\alpha$, then select two tips~$w_1$ and~$w_2$
 with random walks with small~$\alpha$,
but check that 
\[
 \PP^{(t-h)}(w_0)\cup\PP^{(t-h)}(w_1)\cup\PP^{(t-h)}(w_2)
\]
is consistent.

\section{Selfish nodes and Nash equilibria}
\label{s_Nash_theory}
Now, we are going to study the situation when 
some participants of the network are ``selfish''
and want to use a customized attachment strategy,
in order to improve the confirmation time of their
transactions (possibly at the expense of the others).

For a finite set $A$ let us denote by $\M(A)$ the set 
of all probability measures on~$A$,
that is
\[
 \M(A) = \Big\{\mu: A\to\R \text{ such that }\mu(a)\geq 0
  \text{ for all } a\in A \text{ and } \sum_{a\in A}\mu(a)=1\Big\}.
\]
Let 
\[
 \MM = \bigcup_{G=(V,E)\in \G} \M(V\times V)
\]
be the union of the sets of all probability measures
on the pairs of (not necessarily distinct) vertices
of DAGs belonging to~$\G$.
Then, a \emph{general mixed attachment strategy}~$\s$ is a map 
\begin{equation}
\label{gen_att_str}
 \s : \G^{[\dd]} \to \MM
\end{equation}
with the property $\s(V,E,\dd)\in\M(V\times V)$
for any $G^{[\dd]}=(V,E,\dd)\in \G^{[\dd]}$;
that is, for any $G\in \G$ with data attached to the vertices
(which corresponds to the state of the tangle at a given time)
there is a corresponding probability measure on the 
set of pairs of the vertices. 
Note also that in the above we considered \emph{ordered}
pairs of vertices, which, of course, does not restrict 
the generality.

Let $\kappa>0$ be a fixed number.
We now assume that, for a large~$N$, there
are~$\kappa N$ nodes that follow the default tip selection
algorithm, and~$N$ ``selfish'' nodes that 
try to minimize
their ``cost'', whatever it could mean\footnote{for example,
the cost may be the expected confirmation time of a transaction
(conditioned that it is eventually confirmed), the probability
that it was not approved during certain (fixed) time interval, etc.;
below in~\eqref{df_mean_cost} we provide the exact definition
of the cost function we are working with in this paper}.
Assume that all nodes issue transactions with the 
same rate $\frac{\lambda}{(\kappa+1)N}$, independently.
The overall rate of ``honest'' transactions in the system
is then equal to~$\frac{\lambda\kappa}{\kappa+1}$,
and the overall rate of transactions issued
by selfish nodes equals~$\frac{\lambda}{\kappa+1}$.
We also justify the assumption that the number of selfish
nodes is large by observing that
\begin{itemize}
 \item a small number of nodes that do not want to disrupt
the system but just want to obtain some advantages for themselves
(like e.g.\ faster confirmations times) are unlikely to ``globally''
influence the system in any considerable way, even if they do
obtain those advantages for themselves;
 \item however, when it becomes known that it is 
possible to obtain advantages by deviating from the ``recommended''
behavior, it is reasonable to expect that a large number
of independent entities would try to do it.
\end{itemize}

\subsection{Some further assumptions and definitions}
\label{s_restr_str}
Let us now recall that, in practice, the nodes are 
computers running a specialized software, so they are selecting 
the places to attach their transactions in some algorithmic way,
using limited physical resourses. 
In such situation, it is unrealistic to assume 
that a general strategy as in~\eqref{gen_att_str} could be
implemented ``directly'', 
since the space~$\G^{[\dd]}$ is infinite;
for the same reason, even working with \emph{simple} 
attachment strategies (which are maps that take an
element of~$\G^{[\dd]}$ as an input and produce a
\emph{deterministic} pair of its 
vertices as an output) is unrealistic.

Therefore, it looks like a good idea to \emph{restrict}
 the strategy space we are working with.
First,
we consider the following simplifying assumption
(which is, by the way, also quite reasonable, since, in practice,
one would hardly use the genesis as the starting vertex for
 the random walks due to runtime issues):

\medskip
\noindent
\textbf{Assumption L.} There is $n_1>0$ such that the attachment
strategies of all nodes (including those that 
use the default attachment strategy) only depend on the restriction 
of the tangle to the last~$n_1$ transactions that they see.

\medskip 

Observe that, under the above assumption, the set of 
all such strategies can be thought of as a compact convex 
subset of~$\R^d$, where $d=d(n_1)$ is sufficiently large.

In this section we use a different approach to model
the network propagation delays: instead of assuming that 
an incoming transaction does not have information about the 
state of the tangle during last~$h$ units of time, we rather 
assume that it does not have information about the last~$n_0$
transactions attached to the tangle, where~$n_0<n_1$ is some fixed 
positive number (so, effectively, the strategies would depend
on subgraphs induced by $n_1-n_0$ transactions, although the 
results of this section do not rely on this assumption). 
Clearly, these two approaches are quite 
similar in spirit; however, the second one permits us to avoid
certain technical difficulties related to randomness 
of the number of unseen transactions in the first 
case. Also, it will be more natural and convenient 
to pass from continuous to discrete time.

Now, even with the restrictions as above, it is still 
unrealistic to work with the simple strategies of the sort
``choose a fixed pair of transactions for each possible
restriction of the tangle to the set of last~$n_1$
transactions'', because implementing it in practive 
would require effectively dealing with sets indexed by all 
possible restrictions, and the size of the latter set clearly
grows exponentially in~$n_1$. Instead, as hinted in
the beginning of this subsection, we think of different
``attachment methods'' as simple strategies. Formally, 
let $\G^{[\dd]}_{n_1}$ be the set
 of all possible sub-DAGs
of $\G^{[\dd]}$ with~$n_1$ vertices, and~$\MM_{n_1}$
be the set of all probability measures on the vertices' pairs
of elements of~$\G^{[\dd]}_{n_1}$. 
Clearly,
the set~$\G^{[\dd]}_{n_1}$ is finite.
An \emph{attachment method} is then a map 
\[
 \J : \G^{[\dd]}_{n_1} \to \MM_{n_1};
\]
it is thought of as a (randomized)
 polynomial-time polynomial-memory
algorithm
 which takes
the last~$n_1$ transactions and returns a pair of 
those transactions which would serve as attachment's locations.
Then, the 
available simple strategies are attachment methods
\[
 \{ \J_\beta, \beta\in \AAA\},
\]
where $\AAA$ is some (unspecified) index set. 
It is also important to observe that this approach
does not restricts generality.
We then denote by~$\QQ$ the set of all \emph{mixed}
strategies of the form $\J_\Xi$, where~$\Xi$ is a random
variable on~$\AAA$.
Observe also that the set of simple strategies 
can be thought of as a  
subset of~$\R^d$ (which we assume also
to be compact), where $d=d(n_1)$ is sufficiently large,
and~$\QQ$ would be then its convex hull.

Let $\s_1,\ldots,\s_{N}\in\QQ$ be the attachment
strategies used by the selfish nodes.
To evaluate the ``goodness'' of a strategy, one has to choose
and then optimize some suitable observable (that 
stands for the ``cost''); 
as usual, there are
several ``reasonable'' ways to do this. We decided to choose
the following one, for definiteness and
 also for technical reasons (to guarantee
the continuity of a certain function used below); one can probably
extend our arguments to other reasonable cost functions. 
Assume that a transaction~$v$ was attached to the tangle 
at time~$t_v$, so $v\in V_\TT(t)$ for all $t\geq t_v$.
Fix some (typically large) $M_0\in \N$. 
Let $t^{(v)}_1,\ldots,t^{(v)}_{M_0}$
be the moments when the subsequent~$M_0$ (after~$v$) transactions were
attached to the tangle.
For $k=1\ldots,M_0$ 
 let~$R^{(v)}_k$ be the event
that the \emph{default} tip-selecting walk\footnote{i.e., the one
used by nodes following the default attachment strategy} 
on~$\TT\left(t^{(v)}_k\right)$ stops in a tip that \emph{does not} reference~$v$. We then define the random variable
\begin{equation}
\label{df_W^v}
  W(v) = \1{R^{(v)}_1} + \cdots + \1{R^{(v)}_{M_0}}
\end{equation}
to be the number of times that the~$M_0$ ``subsequent''
tip selection random walks do not reference~$v$
(in the above, $\1{A}$ is the indicator function of an event~$A$).
Intuitively, the smaller is the value of~$W(v)/M_0$,
the bigger is the chance that~$v$ is quickly confirmed. 

Next, assume that $(v^{(k)}_j, j\geq 1)$ are the transactions
issued by the $k$th (selfish) node.
We define
\begin{equation}
\label{df_mean_cost}
  \CC^{(k)}(\s_1,\ldots,\s_{N}) 
= M_0^{-1}\lim_{n\to\infty} \frac{W(v^{(k)}_1)+\cdots+W(v^{(k)}_n)}{n}, 
\end{equation}
to be the \emph{mean cost} of the $k$th node 
given that $\s_1,\ldots,\s_{N}$ are the attachment
strategies of the selfish nodes. 

\begin{df}
\label{df_Nash_eq}
We say that a set of strategies $(\s_1,\ldots,\s_{N})\in \QQ^N$
 is a \emph{Nash equilibrium} if 
\[
 \CC^{(k)}(\s_1,\ldots,\s_{k-1},\s_k,\s_{k+1},\ldots,\s_{N})
 \leq 
\CC^{(k)}(\s_1,\ldots,\s_{k-1},\s',\s_{k+1},\ldots,\s_{N}) 
\]
for any~$k$ and any $\s'\in \QQ$.
\end{df}
Observe that, since the nodes are indistinguishable, 
the fact that $(\s_1,\ldots,\s_{N})$ is a Nash equilibrium
implies that so is $(\s_{\sigma_1},\ldots,\s_{\sigma(N)})$
for any permutation~$\sigma$.

\subsection{Main results}
\label{s_mainres}
 From now on, we assume that vertices contain no data,
i.e., the set~$\D$ is empty; this is not absolutely
necessary because, with the data, the proof will be
essentially the same;
however, the notations would become much more cumbersome.
Also, there will be no reattachments;
again, this would unnecessarily complicate the proofs
(one would have to work with \emph{decorated} Poisson processes).
In fact, we are dealing with a so-called \emph{random-turn game}
here, see e.g.\ Chapter~9 of~\cite{KP17} for other examples.

Consider, for the moment, the situation when all nodes
use the same attachment strategy (i.e., there are
no selfish nodes).
The restriction of the tangle on the last~$n_1$
transactions then becomes 
a Markov chain on the state space~$\G_{n_1}$.
We now make the following technical assumption on that 
Markov chain:

\medskip
\noindent
\textbf{Assumption D.} The above Markov chain 
is irreducible and aperiodic.

\medskip 

It is important to observe that Assumption~D is \emph{not}
guaranteed to hold for \emph{every} natural attachment strategy;
however, still, this is not a very restrictive assumption in
practice because every finite Markov chain may be turned 
into an irreducible and aperiodic one by an arbitrarily
small perturbation of the transition matrix.

Then, we are able to prove the following
\begin{theo}
\label{t_Nash_exists} 
Under Assumptions~L and~D, 
the system has at least one Nash equilibrium.
\end{theo}

Symmetric games do not always have symmetric Nash equilibria,
as shown in~\cite{Fey12}. Also, even when such equilibria
exist in the class of mixed strategies, they may be
``inferior'' to asymmetric pure equilibria;
for example, this happens in the classical ``Battle of the sexes''
game (see e.g.\ Section~7.2 of~\cite{KP17}). 


Now, the goal is to prove that, if the number of selfish
nodes~$N$ is large, then for \emph{any} 
equilibrium state the costs of distinct nodes cannot 
be significantly different. 
Let us recall the notations
we use: $\s_1,\ldots,\s_{N}$ are the strategies of the~$N$ 
selfish nodes, and
$\CC^{(k)}(\s_1,\ldots,\s_{N})$, $k=1,\ldots,N$, are the mean costs of
the selfish nodes, defined in~\eqref{df_mean_cost}.
Now, we have the following
\begin{theo}
\label{t_close_equilibr}
For any~$\eps>0$ there exists~$N_0$ (depending on the default
attachment strategy) such that, for all $N\geq N_0$ and 
any Nash equilibrium $(\s_1,\ldots,\s_N)$ it holds that 
\begin{equation}
\label{eq_close_equilibr} 
 \big|\CC^{(k)}(\s_1,\ldots,\s_N) - \CC^{(j)}(\s_1,\ldots,\s_N)\big|
  < \eps
\end{equation}
for all $k,j\in\{1,\ldots,N\}$.
\end{theo}

Now, let us define the notion of \emph{approximate} Nash
equilibrium:
\begin{df}
\label{df_approx_Nash_eq}
For a fixed $\eps>0$,
we say that a set of strategies $(\s_1,\ldots,\s_{N})\in \QQ^N$
 is an $\eps$-equilibrium if 
\[
 \CC^{(k)}(\s_1,\ldots,\s_{k-1},\s_k,\s_{k+1},\ldots,\s_{N})
 \leq 
\CC^{(k)}(\s_1,\ldots,\s_{k-1},\s',\s_{k+1},\ldots,\s_{N}) + \eps
\]
for any~$k$ and any $\s'\in \QQ$.
\end{df}

The motivation for introducing
this notion is that, if~$\eps$ is very small, then, in practice,
$\eps$-equilibria are essentially indistinguishable from the 
``true'' Nash equilibria. 

\begin{theo}
\label{t_symm_equilibr}
For any~$\eps>0$ there exists~$N_0$ (depending on the default
attachment strategy) such that, for all $N\geq N_0$ and 
any Nash equilibrium $(\s_1,\ldots,\s_N)$ it holds that 
$(\s,\ldots,\s)$ is an $\eps$-equilibrium, where
\begin{equation}
\label{eq_symm_equilibr} 
\s=\frac{1}{N}\sum_{k=1}^N \s^{(k)}
\end{equation}
(that is, all selfish nodes use the same ``averaged'' strategy
defined above). The costs of all selfish nodes are then equal
to 
\[
 \frac{1}{N}\sum_{k=1}^N \CC^{(k)}(\s_1,\ldots,\s_N),
\]
that is, the average cost in the Nash equilibrium.
\end{theo}
In other words, for large~$N$ one can essentially assume that 
all selfish nodes follow the same attachment strategy.
This result will be important in Section~\ref{s_simulations},
because it makes it possible to use (practical) simulations
in order to find the Nash equilibria of systems with large
number of selfish players.

\subsection{Proofs}
\label{s_proofs}
First, we need the following technical result:
\begin{lem}
\label{l_continuous}
  Let~$P$ be the transition matrix of an irreducible
and aperiodic discrete-time Markov chain on a finite 
state space~$E$. Let~$\hP$ be a continuous map from a compact
set~$F\subset\R^d$ to the set of all
stochastic matrices on~$E$ (equipped by the distance inherited
 from the usual matrix norm on the space of all matrices on~$E$). 
Fix~$\theta\in(0,1)$, denote 
$\tP(s) = \theta P + (1-\theta)\hP(s)$, and let~$\pi_s$
be the (unique) stationary measure of $\tP(s)$.
Then
$\pi_s$ is also continuous (as a function of~$s$).
\end{lem}

\begin{proof}
In the following we give a (rather) probabilistic proof of this 
fact via the Kac's lemma, although, of course, a purely analytic
proof is also possible.
 Irreducibility and aperiodicity of~$P$ imply that, 
for some $m_0\in \N$ and~$\eps_0>0$
\begin{equation}
\label{irr_aper}
 P_{xy}^{m_0} \geq \eps_0
\end{equation}
for all $x,y\in E$, where $P^{m_0}=(P_{xy}^{m_0}, x,y\in E)$ 
is the transition matrix in~$m_0$ steps.
Now, \eqref{irr_aper} implies that
\begin{equation}
\label{tilde_irr_aper}
 \tP_{xy}^{m_0}(s) \geq \theta^{m_0}\eps_0
\end{equation}
for all $x,y\in E$ and all~$s\in F$.

Being~$(X_n,n\geq 0)$ a stochastic process on~$E$, let us define
\[
 \tau(x) = \min\{k\geq 1: X_k = x\}
\]
(with the convention $\min\emptyset = \infty$) to be the 
\emph{hitting time} of the site~$x\in E$ by the stochastic 
process~$X$. Now, let $\IP_x^{(s)}$ and~$\IE_x^{(s)}$
be the probability and the expectation with respect to the 
Markov chain with transition matrix~$\tP(s)$ starting from~$x\in E$.
We now recall the Kac's lemma (cf. e.g.\ Theorem~1.22 
of~\cite{Durrett12}): for all~$x\in E$ it holds that 
\begin{equation}
\label{eq_Kac}
 \pi_s(x) = \frac{1}{\IE_x^{(s)}\tau(x)} .
\end{equation}
Now, \eqref{tilde_irr_aper} readily implies that, for all $x\in E$
and $n\in\N$,
\begin{equation}
\label{tail_tau}
 \IP_x^{(s)}[\tau(x)\geq n] \leq c_1 e^{-c_2 n}
\end{equation}
for some positive constants~$c_{1,2}$ which do not depend on~$s$.
This in its turn implies that the series
\[
 \IE_x^{(s)}\tau(x) = \sum_{n=1}^\infty \IP_x^{(s)}[\tau(x)\geq n]
\]
converges uniformly in~$s$ and so $\IE_x^{(s)}\tau(x)$ 
is uniformly bounded from above\footnote{and, of course, it is also
bounded from below by~$1$}; 
also, the Uniform Limit Theorem (see e.g.\ Section~D.6.2
of~\cite{OK07}) implies 
that $\IE_x^{(s)}\tau(x)$ is continuous in~$s$.
Therefore,
for any~$x\in E$, \eqref{eq_Kac} implies 
that $\pi_s(x)$ is also a continuous function of~$s$.
\end{proof}

\begin{proof}[Proof of Theorem~\ref{t_Nash_exists}.]
 The authors were unable to find a result available in the literature
that implies Theorem~\ref{t_Nash_exists} directly; nevertheless,
its proof is quite standard and essentially follows 
Nash's original paper~\cite{Nash} (see also~\cite{F64}).
There is only one technical difficulty, which we intend
to address via the above preparatory steps: one needs
to prove the continuity of the cost function.

Denote by $\pi_\s$ the invariant measure of the Markov
chain given that the (selfish) nodes use the ``strategy vector''
$\sss=(\s_1,\ldots,\s_N)$. 
Then, the idea is to use Lemma~\ref{l_continuous} with
 $\theta=\frac{\kappa}{\kappa+1}$, $P$ the transition
matrix obtained from the default attachment strategy,
and~$\hP(s)$ is the transition
matrix obtained from the strategy $\s'=N^{-1}\sum_{k=1}^N \s_k$
(observe that $N$ nodes using the strategies $\s_1,\ldots,\s_N$,
is the same as one node with strategy~$\s'$ issuing
transactions~$N$ times faster).
Assumption~D together with Lemma~\ref{l_continuous} then
imply that~$\pi_\sss:=\pi_{\s'}$ 
is a continuous function of~$\sss$.

Let $\IE^{\s, \hat\s}_{\pi_{\s'}}$ be the expectation
with respect to the following procedure:
take the ``starting'' graph according to~$\pi_{\s'}$,
then attach to it a transaction according to the strategy~$\s$,
and then keep attaching subsequent transactions 
according to the strategy~$\hat\s$ (instead of $\s'$ and~$\hat\s$
we may also use the strategy vectors; $\s'$ and~$\hat\s$ would be 
then their averages).
Let also~$W^{(k)}$ be the random variable defined
as in~\eqref{df_W^v} for an arbitrary transaction~$v$ issued by the
$k$th node. 
Then, the Ergodic Theorem for Markov chains (see e.g.\
Theorem~1.23 of~\cite{Durrett12}) implies that
\begin{equation}
 \CC^{(k)}(\s) = \IE^{\s_k,\s'}_{\pi_{\s'}} W^{(k)}.
\end{equation}
It is not difficult to see that the above expression
is a polynomial of the $\s$'s coefficients (i.e., the 
corresponding probabilities) and $\pi_{\s'}$-values,
and hence it is a continuous function on the space 
of strategies~$\MM_{n_1}$.
Using this, the rest of the proof is standard,
it is obtained as a consequence of the Kakutani's
fixed point theorem~\cite{Kakutani}, also with the help
of the Berge's Maximum Theorem
(see e.g.\ Chapter~E.3 of~\cite{OK07}).
\end{proof}

\begin{proof}[Proof of Theorem~\ref{t_close_equilibr}.]
Without restricting generality we may assume that
\begin{align*}
 \CC^{(1)}(\s_1,\ldots,\s_N) &= 
   \max_{k=1,\ldots,N}\CC^{(k)}(\s_1,\ldots,\s_N),\\
\CC^{(2)}(\s_1,\ldots,\s_N) &= 
   \min_{k=1,\ldots,N}\CC^{(k)}(\s_1,\ldots,\s_N),
\end{align*}
so we then need to proof that $\CC^{(1)}(\sss)-\CC^{(2)}(\sss)<\eps$,
where $\sss=(\s_1,\ldots,\s_N)$.
Now, the main idea of the proof is the following: if $\CC^{(1)}(\sss)$ is considerably larger than $\CC^{(2)}(\sss)$, then the owner of 
the first node may decide to adopt the strategy used by the second
one. This would not necessarily decrease his costs to the former
costs of the second node since a change in an individual
strategy leads to changes in \emph{all} costs; however, 
when~$N$ is large, the effects of changing the strategy
of only one node would be small, and (if the difference
of $\CC^{(1)}(\sss)$ and $\CC^{(2)}(\sss)$ were not small)
this would lead to a contradiction to the assumption
that~$\sss$ was a Nash equilibrium.

So, let us denote $\sss'=(\s_2,\s_2,\s_3,\ldots,\s_N)$,
the strategy vector after the first node adopted the strategy
of its ``more successful'' colleague,
see Figure~\ref{f_adopt_str}. 
\begin{figure}
 \centering \includegraphics{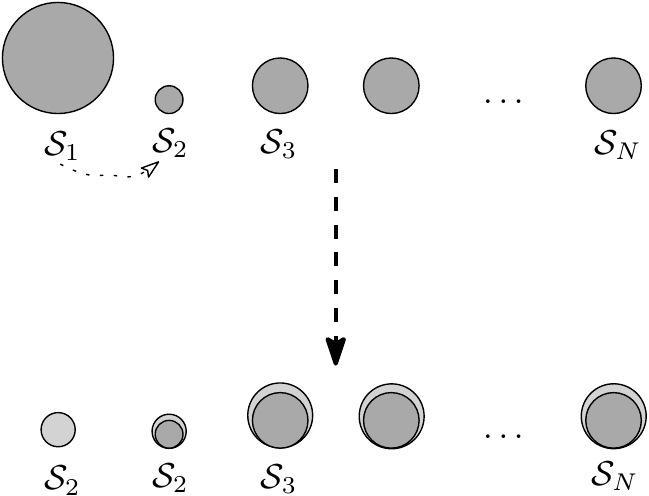} 
\caption{On the main idea of the proof 
of Theorem~\ref{t_close_equilibr}. The node with the highest
cost will switch to the strategy of the node with the lowest
cost. That will not guarantee exactly that same cost to the 
former node, but the difference will be rather small
since~$N$ is large (so the change in one component of
the strategy vector will not influence a lot the outcome).}
\label{f_adopt_str}
\end{figure}
Let
\[
 \s = \frac{1}{N}\big(\s_1+\cdots +\s_N\big)
 \text{ and }\s' = \frac{1}{N}\big(2\s_2+\s_3+\cdots +\s_N\big)
\]
be the two ``averaged'' strategies.
In the following, we are going
 to compare $\CC^{(2)}(\sss)=\IE^{(\s_2,\s)}_{\pi_{\s}} W^{(2)}$
(the ``old'' cost of the second node) 
with $\CC^{(1)}(\sss')=\IE^{(\s_2,\s')}_{\pi_{\s'}} W^{(1)}$
(the ``new'' cost of the first node, after it adopted
the second node's strategy).
We need the following
\begin{lem}
\label{l_1st_node}
 For any measure~$\pi$ on~$\G_{n_1}$ and any  strategy 
vectors $\sss=(\s_1,\ldots,\s_N)$ and $\sss'=(\s'_1,\ldots,\s'_N)$
such that $\s_k=\s'_k$ for all $k=2,\ldots,N$, we have
\begin{equation}
\label{eq_1st_node} 
\big|\IE^{(\s_j,\s)}_{\pi} W^{(j)}-\IE^{(\s'_j,\s')}_{\pi} W^{(j)}\big|
  \leq \frac{M_0}{N}
\end{equation}
for all $j=2,\ldots,N$.
\end{lem}
\begin{proof}
Let us define the event
\[
 A = \{\text{among the $M_0$ transactions there is 
at least one issued by the first node} \},
\]
and observe that, by the union bound, the probability that 
it occurs is at most~$M_0/N$.
Then, using the fact that 
$\IE^{(\s_j,\s)}_{\pi} (W^{(j)}\1{A^c})
=\IE^{(\s_j,\s')}_{\pi} (W^{(j)}\1{A^c})$
(since, on~$A^c$, the first node does not ``contribute''
to~$W^{(j)}$), write
\begin{align*}
\lefteqn{
\big|\IE^{(\s_j,\s)}_{\pi} W^{(j)}-\IE^{(\s'_j,\s')}_{\pi} W^{(j)}\big|
}
\\
&= \big|\IE^{(\s_j,\s)}_{\pi} (W^{(j)}\1{A})
+\IE^{(\s_j,\s)}_{\pi} (W^{(j)}\1{A^c})
-\IE^{(\s_j,\s')}_{\pi} (W^{(j)}\1{A})
-\IE^{(\s_j,\s')}_{\pi} (W^{(j)}\1{A^c})\big|\\
&= \big|\IE^{(\s_j,\s)}_{\pi} (W^{(j)}\1{A})
-\IE^{(\s_j,\s')}_{\pi} (W^{(j)}\1{A})\big|\\
&\leq \frac{M_0}{N},
\end{align*}
where we also used that $W^{(j)}\leq 1$. 
This concludes the proof of Lemma~\ref{l_1st_node}.
\end{proof}

We continue proving Theorem~\ref{t_close_equilibr}.
First, by symmetry, we have
\begin{equation}
\label{2.4_1}
  \IE^{(\s_2,\s')}_{\pi_{\s'}} W^{(1)}
 = \IE^{(\s_2,\s')}_{\pi_{\s'}} W^{(2)}.
\end{equation}
Also, it holds that
\begin{equation}
\label{2.4_2}
 \big|\IE^{(\s_2,\s')}_{\pi_{\s'}} W^{(2)} 
    - \IE^{(\s_2,\s)}_{\pi_{\s'}} W^{(2)} \big|
\leq \frac{M_0}{N}
\end{equation}
by Lemma~\ref{l_1st_node}.
Then, similarly to the proof of Theorem~\ref{t_Nash_exists},
we can obtain that the function
\[
 (\s,\s',\s'') \mapsto \IE^{(\s,\s')}_{\pi_{\s''}} W^{(2)}
\]
is continuous; since it is defined on a compact, 
it is also uniformly continuous. That is,
for any~$\eps'>0$ there exist~$\delta'>0$ such that 
if $\|(\s,\s',\s'')-({\tilde\s},{\tilde\s}',{\tilde\s}'')\|<\delta'$,
 then
\[
 \big|\IE^{(\s,\s')}_{\pi_{\s''}} W^{(2)}
  - \IE^{({\tilde\s},{\tilde\s}')}_{\pi_{{\tilde\s}''}} W^{(2)} \big|
      < \eps'.
\]
Choose~$N_0=\lceil1/\delta' \rceil$. We then obtain from the above 
that
\begin{equation}
\label{2.4_3}
 \big|\IE^{(\s_2,\s)}_{\pi_{\s'}} W^{(2)} 
  - \IE^{(\s_2,\s)}_{\pi_{\s}} W^{(2)}  \big|
      < \eps'.
\end{equation}

 The relations~\eqref{2.4_1}, \eqref{2.4_2}, and~\eqref{2.4_3}
imply that
\[
  \big| \IE^{(\s_2,\s')}_{\pi_{\s'}} W^{(1)}    
-  \IE^{(\s_2,\s)}_{\pi_{\s}} W^{(2)}
 \big|
 \leq \eps' + \frac{M_0}{N}.
\]
On the other hand, since we assumed that~$\sss$ is a Nash equilibrium,
it holds that 
\begin{equation}
\label{eq***N_eq}
 \IE^{(\s_2,\s')}_{\pi_{\s'}} W^{(1)} = \CC^{(1)}(\sss') \geq
   \CC^{(1)}(\sss) = \IE^{(\s_1,\s)}_{\pi_{\s}} W^{(1)},
\end{equation}
which implies that
\[
  \IE^{(\s_1,\s)}_{\pi_{\s}} W^{(1)}
- \IE^{(\s_2,\s)}_{\pi_{\s}} W^{(2)} \leq \eps' + \frac{M_0}{N}.
\]
This concludes the proof of Theorem~\ref{t_close_equilibr}.
\end{proof}

\begin{proof}[Proof of Theorem~\ref{t_symm_equilibr}.]
To begin, we observe that the proof of the second part is immediate,
since, as already noted before, for an external observer,
the situation where there are~$N$ nodes with strategies
$(\s_1,\ldots,\s_N)$ is indistinguishable from the situation
with one node with averaged strategy. 

Now, we need to prove that, for any fixed~$\eps'>0$ it holds that
\begin{equation}
\label{eps_eq_1}
  \CC^{(1)}(\s,\ldots,\s) \leq 
\CC^{(1)}({\tilde\s},\s,\ldots,\s) + \eps'
\end{equation}
for all large enough~$N$ 
(the claim would then follow by symmetry).
Recall that we have 
\begin{align}
\label{d1}
\CC^{(1)}(\s,\ldots,\s)&=\IE^{(\s,\s)}_{\pi_{\s}} W^{(1)},\\
\label{d2}
\CC^{(1)}(\s_1,\ldots,\s_N)&=\IE^{(\s_1,\s)}_{\pi_{\s}} W^{(1)},\\
\intertext{and}
\label{d3}
\CC^{(1)}({\tilde\s},\s,\ldots,\s)
&= \IE^{({\tilde\s},\s')}_{\pi_{\s'}} W^{(1)},
\end{align}
where
\[
\s'=\frac{1}{N}\big({\tilde\s}+(N-1)\s\big)
 = \frac{1}{N}\Big({\tilde\s}+\frac{N-1}{N}
 (\s_1+\cdots+\s_N)\Big).
\]

Now, the second part of this theorem together with
 Theorem~\ref{t_close_equilibr} imply\footnote{note that 
Theorem~\ref{t_close_equilibr} implies that, when~$N$ is large,
the nodes already have ``almost'' the same cost 
in the Nash equilibrium $(\s_1,\ldots,\s_N)$}
 that, for any fixed $\eps>0$
\begin{equation}
\label{close_average}
  \big|\IE^{(\s,\s)}_{\pi_{\s}} W^{(1)}
     - \IE^{(\s_1,\s)}_{\pi_{\s}} W^{(1)}\big| < \eps
\end{equation}
for all large enough~$N$. 

Next, let us denote
\[
 \s'' = \frac{1}{N}({\tilde\s}+\s_2+\cdots+\s_N).
\]
Then, again using the uniform continuity argument
(as in the proof of Theorem~\ref{t_close_equilibr}),
we obtain that, for any~$\eps''>0$
\begin{equation}
\label{eps''}
 \big| \IE^{({\tilde\s},\s')}_{\pi_{\s'}} W^{(1)}
    - \IE^{({\tilde\s},\s'')}_{\pi_{\s''}} W^{(1)}\big|
     < \eps''
\end{equation}
for all large enough~$N$. However,
\[ 
\IE^{({\tilde\s},\s'')}_{\pi_{\s''}} W^{(1)}
 = \CC^{(1)}({\tilde\s},\s_2,\ldots,\s_N)
 \geq \CC^{(1)}(\s_1,\s_2,\ldots,\s_N)
=\IE^{(\s_1,\s)}_{\pi_{\s}} W^{(1)},
\]
since $(\s_1,\ldots,\s_N)$ is a Nash equilibrium.
 Then, \eqref{close_average}--\eqref{eps''}
imply that
\[
  \big|\IE^{(\s,\s)}_{\pi_{\s}} W^{(1)}
 -  \IE^{({\tilde\s},\s')}_{\pi_{\s'}} W^{(1)} \big| 
 <\eps + \eps'',
\]
and, recalling~\eqref{d1} and~\eqref{d3}, 
we conclude the proof of Theorem~\ref{t_symm_equilibr}.
\end{proof}

\section{Simulations}
\label{s_simulations}
In this section we investigate Nash
equilibria between selfish nodes via simulations.
As discussed in Section~\ref{s_intro},
this is motivated by the following important question:
since the choice of an attachment strategy is not enforced,
there may indeed be nodes which would prefer
to ``optimise'' their strategies in order to 
decrease the mean confirmation time of their transactions.
So, can this lead to a situation where the corresponding
Nash equilibrium is ``bad for everybody'', effectively
leading to the system's malfunctioning?

Due to Theorem~\ref{t_symm_equilibr} 
we may assume that all selfish nodes
use the same attachment strategy. Even then,
it is probably unfeasible to calculate that 
strategy exactly; instead, we resort to simulations,
which indeed will show that the equilibrium strategy 
of the selfish
nodes will not be much different from the 
(suitably chosen) default strategy, at least in the (very natural)
situation below. 
But, before doing that, let us explain the intuition behind
this fact. Naively, a reasonable strategy for a selfish node
would be the following:
\begin{itemize}
 \item[(1)] Calculate the exit distribution of the tip-selecting
 random walk.
 \item[(2)] Find the two tips where this distribution attains
its ``best''\footnote{i.e., the maximum and the second-to-maximum}
values.
 \item[(3)] Approve these two tips.
\end{itemize}

\begin{figure}
 \centering \includegraphics[width=\textwidth]{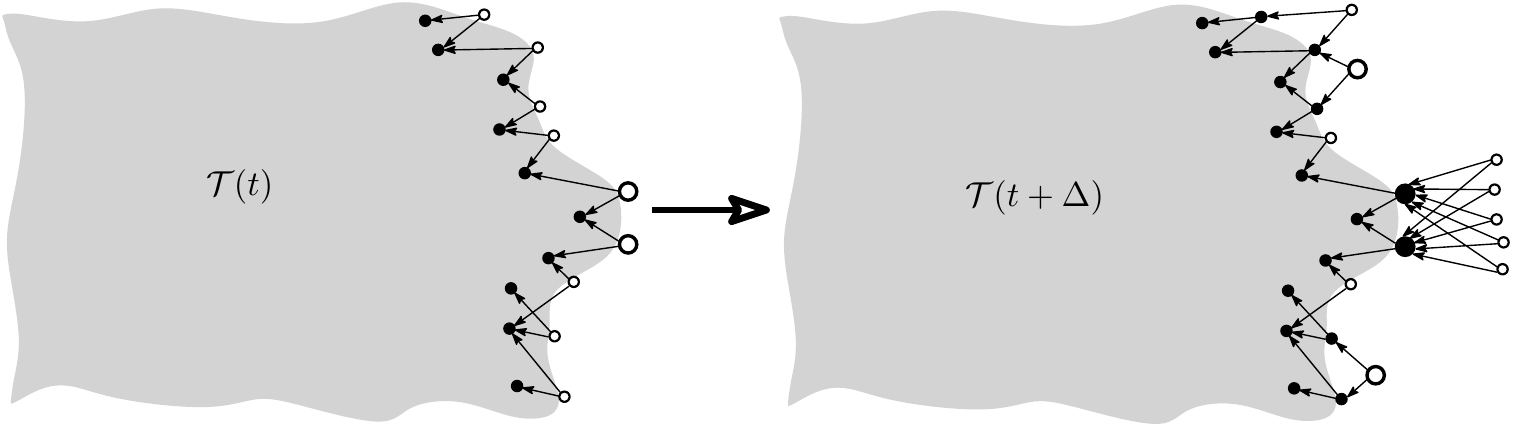} 
\caption{Why the ``greedy'' tip selection strategy will not work
(the two ``best'' tips are shown as larger circles).}
\label{f_greedy}
\end{figure}
However, this strategy fails when other selfish nodes 
are present. To understand this, look at Figure~\ref{f_greedy}:
\emph{many} selfish nodes attach their transactions
to the two ``best'' tips. As a result, the ``neighborhood'' of
these two tips 
becomes ``overcrowded'': there is so much competition 
between the transactions issued by the selfish nodes,
that the chances of them being approved soon actually
decrease\footnote{the ``new'' best tips are not among
them, as shown on Figure~\ref{f_greedy} on the right}.

To illustrate this fact, several simulations have been done. 
All the results depicted here were generated 
using~\eqref{trans_probs} as the transition probabilities,
 with $q = 1/3$, and a network delay of $h=1$ second. 
Also, a transaction will be reattached if the two 
following criteria are met:
\begin{itemize}
 \item[(1)] the transaction is older than 
20~seconds
 \item[(2)] the transaction is not referenced by the 
tip selected by a random walk with $\alpha = \infty$\footnote{here,
 when the random walk must choose among $n$ transactions with 
the same weight, it will choose randomly, with equal probabilities}.
\end{itemize}
 
 This way, we guarantee not only that the unconfirmed transactions 
will be eventually confirmed, but also that all transactions that 
were never reattached are referenced by most of the tips. 
Note that when the reattachment is allowed in the simulations,
 if a new transaction references an old, already reattached 
transaction together with its newly reissued counterpart, 
there will be a double 
spending. Even though the odds of that are low (since when 
a transaction is re-emitted, it will be old enough to be almost 
never chosen by the random walk algorithm), a specific procedure 
was included in the simulations in order to not allow double spendings.

The average costs were simulated as defined at equations~\eqref{df_W^v} and~\eqref{df_mean_cost}, so a certain value of $M_0$ had to be chosen. Since the value of $W(v)/\lambda$ is related to the time of approval of $v$ (whenever the transaction is indeed approved before $t^{(v)}_{M_0}$), we want $M_0$ to be sufficiently large, in order to capture the effect of most of the approvals. Figure~\ref{fig.conv} depicts the typical cumulative distribution of the time of the first approval, for several values of $\alpha$ and $\lambda$. Note that roughly 95$\%$ of the transactions will be approved before $t=5$s, and almost its totality will be approved before $t=10$s. For that reason, in both cases ($\lambda=25$ and $\lambda=50$), the mean cost was calculated over the transactions attached during
a time interval of approximately 10s ($M_0=500$ for $\lambda=50$ 
and $M_0=250$ for $\lambda=25$), so almost the totality of approvals will be ``seen'' by the average cost.

\begin{figure}[ht]
\centering \includegraphics[width=14cm]{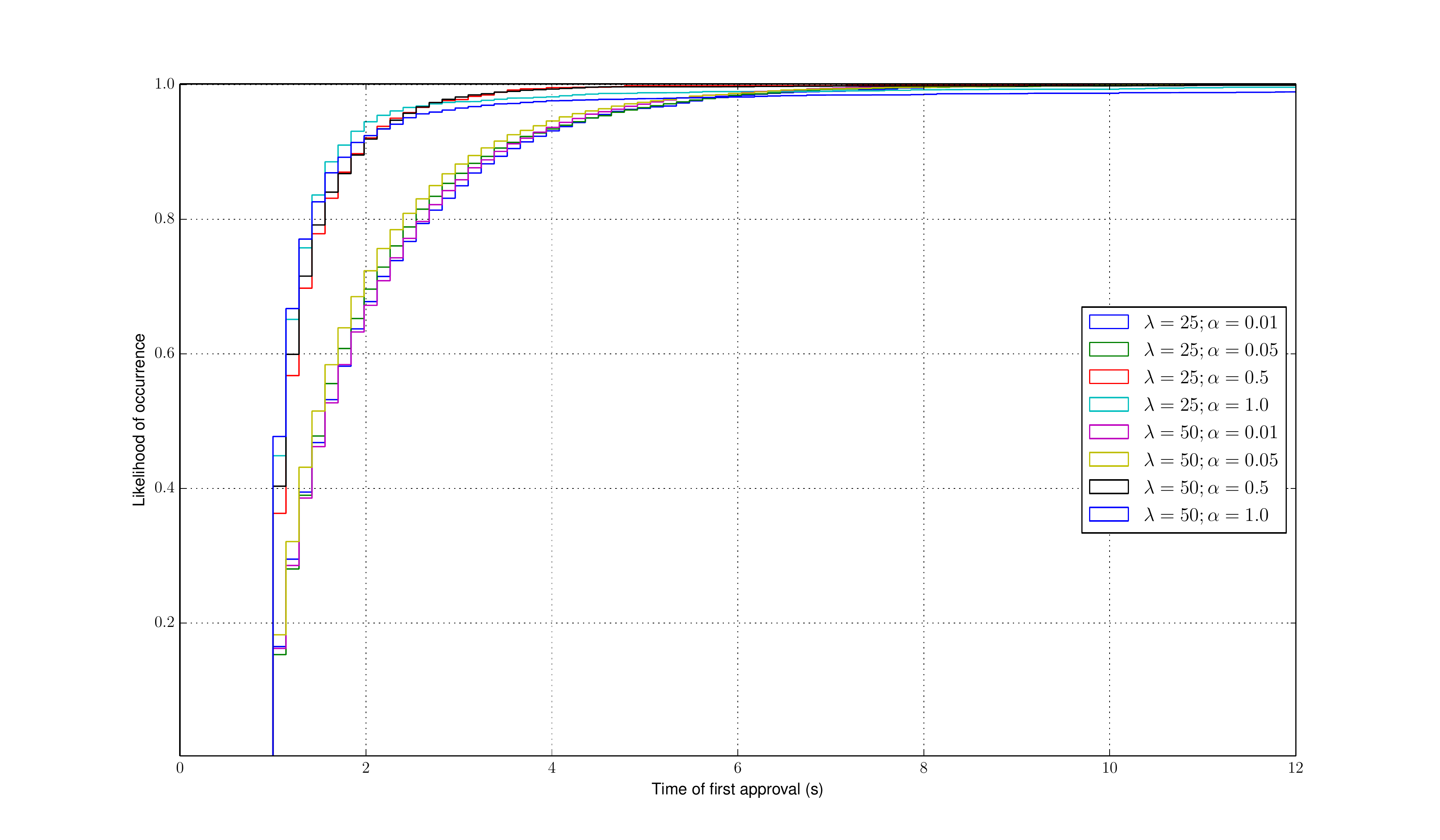} 
\caption{Cumulative distribution of time of approvals for several values of $\alpha$ and $\lambda$}
\label{fig.conv}
\end{figure}

\subsection{One dimensional Nash equilibria}
\label{s_1dim_Nash}
In this section, we will study the Nash equilibria
$(\s_1,\ldots,\s_N)$ of 
the tangle problem, considering the following strategy space:
\[
\big\{(1-\theta)\s^0 + \theta\s^1, \ 0\leq \theta \leq 1\big\}
\]
 where the simple strategies~$\s^0$ and~$\s^1$
are the default tip selection strategy
and the ``greedy'' strategy (defined in the beginning of this 
section) correspondingly;
that is, $\s_i=(1-\theta_i)\s^0 + \theta_i\s^1$ where
 $\theta_i \in [0,1]$,
$i=1,\ldots,N$. The goal is to find the Nash equilibria 
relative to the costs defined in the last section 
(equations~\eqref{df_mean_cost} and~\eqref{df_W^v}). 
The selfish nodes will 
try to optimise their transaction cost with respect to~$\theta_i$.

By Theorem~\ref{t_symm_equilibr}, each Nash equilibrium in this 
form will be equivalent to another Nash equilibrium with 
``averaged'' strategies, i.e.:
\[
\s=\left(1-\frac{1}{N}\sum_{k=1}^N \theta_i \right)\s^0 +
\frac{1}{N}\sum_{k=1}^N \theta_i \s^1=(1-\theta)\s^0 + \theta\s^1 \hspace{1cm} 
\text{for each } i=1,\ldots,N,
\]
Now, suppose that we have a fixed fraction~$\gamma$ of selfish 
nodes, that choose a strategy among the possible~$\s$. 
The non-selfish nodes will not be able to choose their strategy,
 so they will be restricted, as expected, to~$\s^0$. Note that, 
since they cannot choose their strategy, they will not ``play'' 
the game. Since the costs are linear over $\s$, such mixed strategy 
game will be equivalent\footnote{this way, we deal with just one variable ($p$) instead of two ($\gamma$ and $\theta$) and none of the parameters of the system is lost} to 
a game where only a fraction
 $p=\gamma\theta\leq\gamma$ of the nodes 
chooses~$\s^1$ over~$\s^0$, and the rest of the nodes 
chooses~$\s^0$ over~$\s^1$.  Note that this equivalence does not contradict the theorems proved in the last sections, that state:
\begin{itemize}
\item all the nodes will have the same average costs when the system is at a Nash equilibrium;
\item any Nash equilibrium has an equivalent Nash equilibrium with ``averaged'' strategies, where all the nodes will have the same strategies.
\end{itemize}

 From now on, we will refer (unless stated otherwise) 
to this second pure strategy game. 
Figure~\ref{a001}(a) represents a typical graph of average costs 
of transactions issued under~$\s^0$ and~$\s^1$, as a function 
of the fraction~$p$, for a low~$\alpha$ and two different 
values of~$\lambda$. As already demonstrated, when in equilibrium, 
the selfish nodes should issue transactions with the same average costs.
 That means that the system should reach equilibrium 
in one of the following states:
\begin{itemize}
 \item[(1)] some selfish nodes choose~$\s^0$ 
    and the rest choose $\s^1$ ($0<p<\gamma$), 
all of them with the same average costs;
 \item[(2)] all selfish nodes choose~$\s^1$ ($p=\gamma$);
 \item[(3)] all selfish nodes choose~$\s^0$ ($p=0$).
\end{itemize}

\begin{figure}[ht]
\hspace{4cm} (a) \hspace{7.4cm} (b)\\
\includegraphics[width=7.4cm,height=6.6cm]{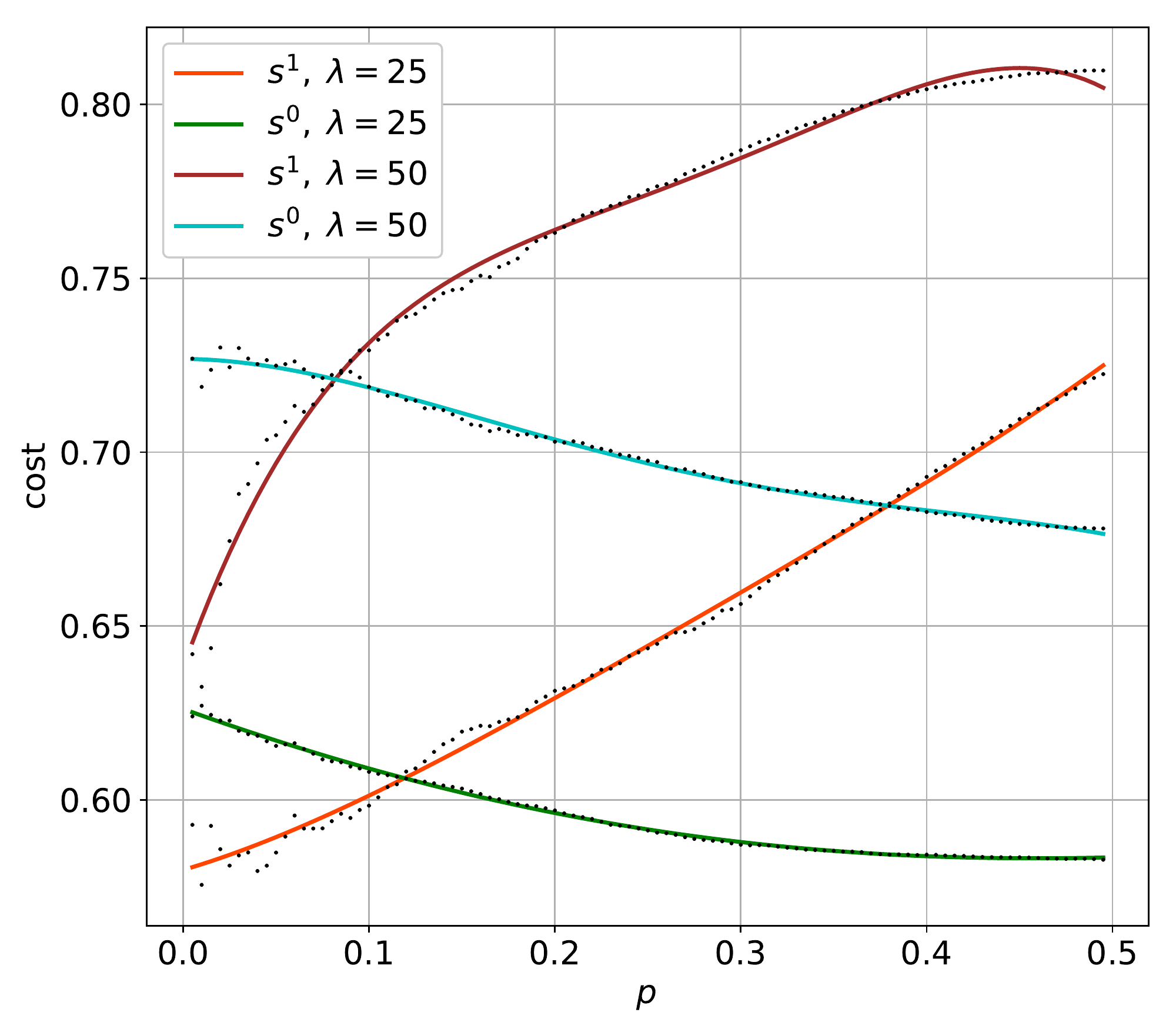} 
\includegraphics[width=7.4cm,height=6.6cm]{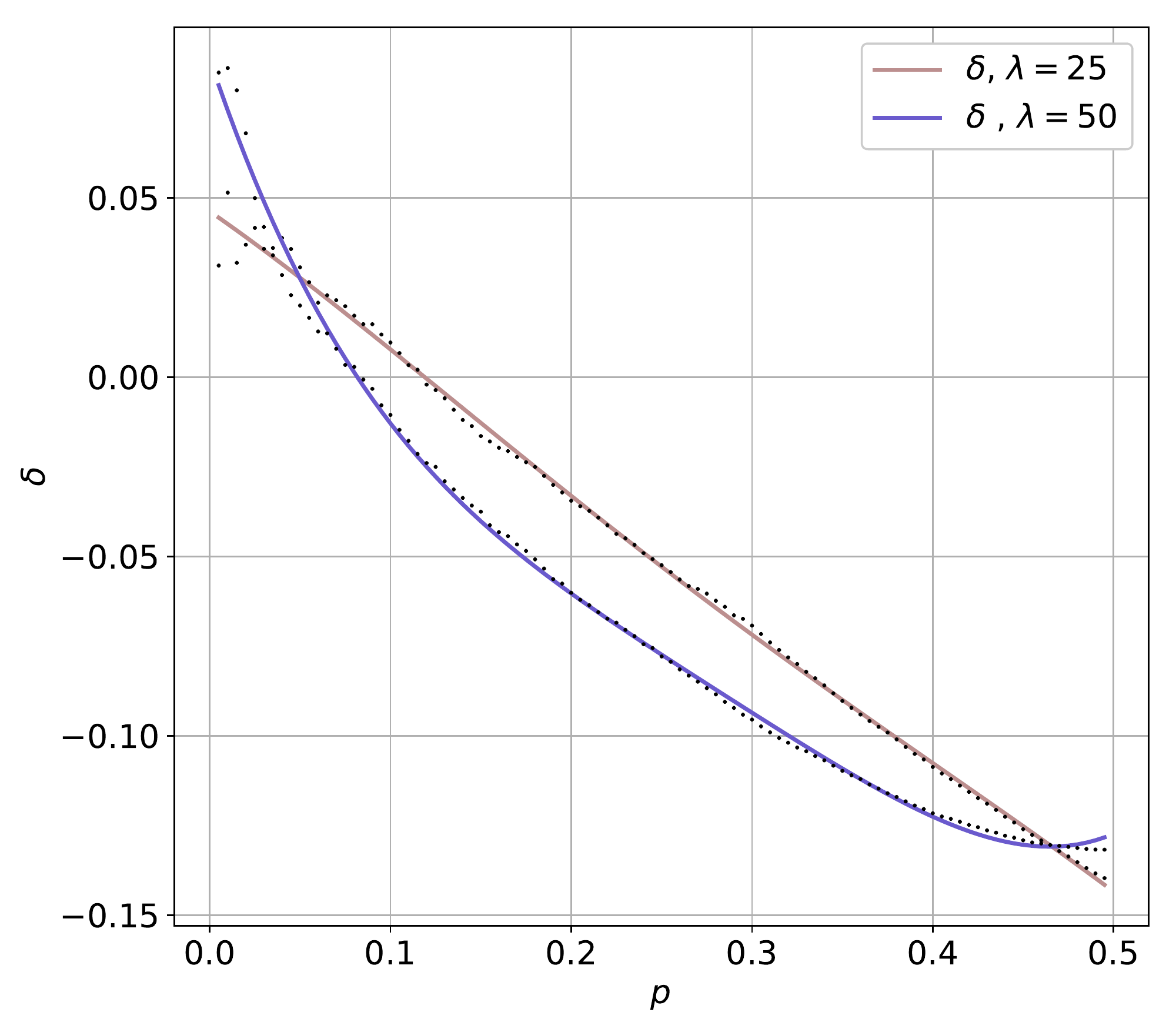}
\caption{Dotted lines are the raw data. Solid lines were fitted with least 
squares polynomials of four-degree. Costs (a) and gain of the 
strategy $\s^1$ over $\s^0$; (b) for $\alpha=0.01$.}
\label{a001}
\end{figure}

If the two curves on the graphs do not intersect, 
the equilibrium should be clearly at state~(2) or~(3), 
depending on which of the average costs is larger. 
If the two curves on the graphs intercept each other,
we will also have the intersection point as a Nash equilibrium
 candidate. We call~$\bar{\sss}$ the vector of strategies 
on equilibrium and~$\bar{p}$ the fraction of nodes that will 
issue transactions under~$\s^1$ when the system is in~$\bar{\sss}$. 
We define $p^{-}=\bar{p}-\frac{\gamma}{N}$ and
 $p^{+} = \bar{p}+\frac{\gamma}{N}$, meaning that~$p^{-}$ and~$p^{+}$
 will be deviations from~$\bar{p}$, that result from one node
 switching strategies, from~$\s^0$ to~$\s^1$ and from~$\s^1$ to~$\s^0$,
 respectively. We also define~$\bar{\sss}^{-}$ and~$\bar{\sss}^{+}$ 
as strategy vectors related to~$p^{-}$ and~$p^{+}$. 
Note on Figure~\ref{fig.stability} that this kind of Nash equilibrium
 candidate may not be a real equilibrium. 
In the first example (\ref{fig.stability}(a)), when the system is 
at point~$\bar{p}$ and a node switches strategies from~$\s^0$ 
to~$\s^1$ (moving from $\bar{p}$ to $p^+$), the cost actually decreases, so~$\bar{p}$ cannot be 
a Nash equilibrium. On the other hand, the second example
 (\ref{fig.stability}(b)) shows a Nash equilibrium 
at point~$\bar{p}$, since deviations to~$p^{-}$ and~$p^{+}$
 will increase costs. 

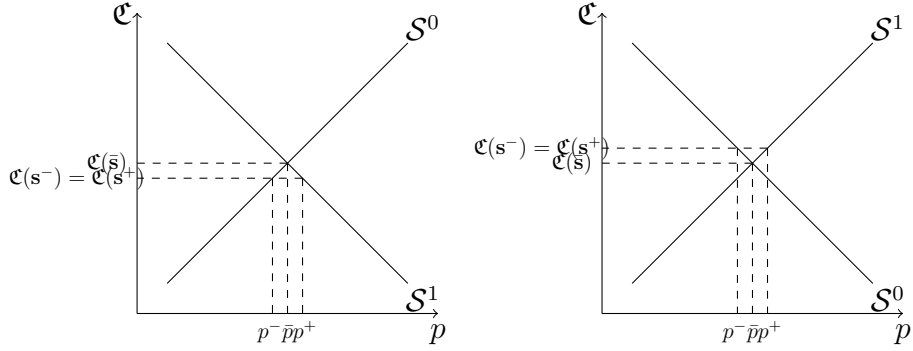
\begin{figure}
\begin{tikzpicture}[scale=0.4, auto]
\draw[arrows=->] (0,0)--(0,10);
\node[align=left] at (-0.5,10) {$\CC$};
\draw[arrows=->] (0,0)--(10,0);
\node[scale=1,align=left,below] at (10,0) {$p$};
\draw (1,1)--(9,9);
\node at (9.5,9.5) {$\s^0$};
\draw (1,9)--(9,1);
\node at (9.5,0.5) {$\s^1$};
\draw[dashed] (4.5,4.5)--(4.5,0);
\draw[dashed] (5,5)--(5,0);
\node[scale=0.7,align=left,below] at (5,0) {$p^-\bar{p}p^+$};
\draw[dashed] (5.5,4.5)--(5.5,0);
\draw[dashed] (5,5)--(0,5);
\node[scale=0.7,align=right] at (-1,5) {$\CC(\bar{\sss})$};
\draw[dashed] (5.5,4.5)--(0,4.5);
\node[scale=0.7,align=right] at (-2,4.5) {$\CC(\sss^{-})=\CC(\sss^{+})$};
\end{tikzpicture}
\begin{tikzpicture}[scale=0.4, auto]
\draw[arrows=->] (0,0)--(0,10);
\node[align=left] at (-0.5,10) {$\CC$};
\draw[arrows=->] (0,0)--(10,0);
\node[align=left, below] at (10,0) {$p$};
\draw (1,1)--(9,9);
\node at (9.5,9.5) {$\s^1$};
\draw (1,9)--(9,1);
\node at (9.5,0.5) {$\s^0$};
\draw[dashed] (4.5,5.5)--(4.5,0);
\draw[dashed] (5,5)--(5,0);
\node[scale=0.7,align=left,below] at (5,0) {$p^-\bar{p}p^+$};
\draw[dashed] (5.5,5.5)--(5.5,0);
\draw[dashed] (5,5)--(0,5);
\node[scale=0.7,align=right] at (-1,5) {$\CC(\bar{\sss})$};
\draw[dashed] (5.5,5.5)--(0,5.5);
\node[scale=0.7,align=right] at (-2,5.5) {$\CC(\sss^{-})=\CC(\sss^{+})$};
\end{tikzpicture}
\caption{Different Nash equilibrium points in systems 
with similar curves}
\label{fig.stability}
\end{figure}

Now, let us re-examine Figure~\ref{a001}(a). 
Here, the Nash equilibrium will occur at the point~$\bar{p}$, 
since we have a situation as on Figure~\ref{fig.stability}(b). 
That point is easily found at Figure~\ref{a001}(b), 
when $\delta=0$. Note that the Nash equilibrium for a larger~$\lambda$
 will be at a smaller~$\theta_0$ than the Nash equilibrium 
for a smaller~$\lambda$. This was already expected, since, 
for a larger~$\lambda h$, the tips will be naturally 
more ``overcrowded'', so the effect depicted at 
Figure~\ref{f_greedy} will be amplified. 
Thus, the Nash equilibrium for the higher~$\lambda h$ cases 
must occur with a smaller proportion of transactions 
issued with the pure strategy~$\s^1$.

Let us now again
consider the mixed strategy game. In the case when all 
the nodes are allowed to choose between the two pure strategies 
($\s^0$ and $\s^1$), the Nash equilibrium will be indeed 
at $\theta_0 = \bar{p}$ (as expected, since in this 
case $\gamma=1$). If just a fraction $\gamma = p/\theta >\bar{p}$ 
of the nodes is selfish, then the Nash equilibrium will occur when $\theta_0 = \bar{p}/\gamma$. 
Now, if $\gamma\leq \bar{p}$, the costs of the nodes will not 
coincide\footnote{that is the case for the range of studied 
parameters}. 
In that case, the average cost of transactions under~$\s^1$
will always be smaller than the average cost of transactions 
under~$\s^0$,  meaning that the Nash equilibrium will be met 
at $\theta_0 = 1$. Summing up, the Nash equilibrium $\theta_0$, 
in these cases, will be met at:
\[
\theta_0 = \min\{\bar{p}/\gamma,1\}.
\]

\begin{figure}
\hspace{4cm} (a) \hspace{7.4cm} (b)\\
\includegraphics[width=7.4cm,height=6.6cm]{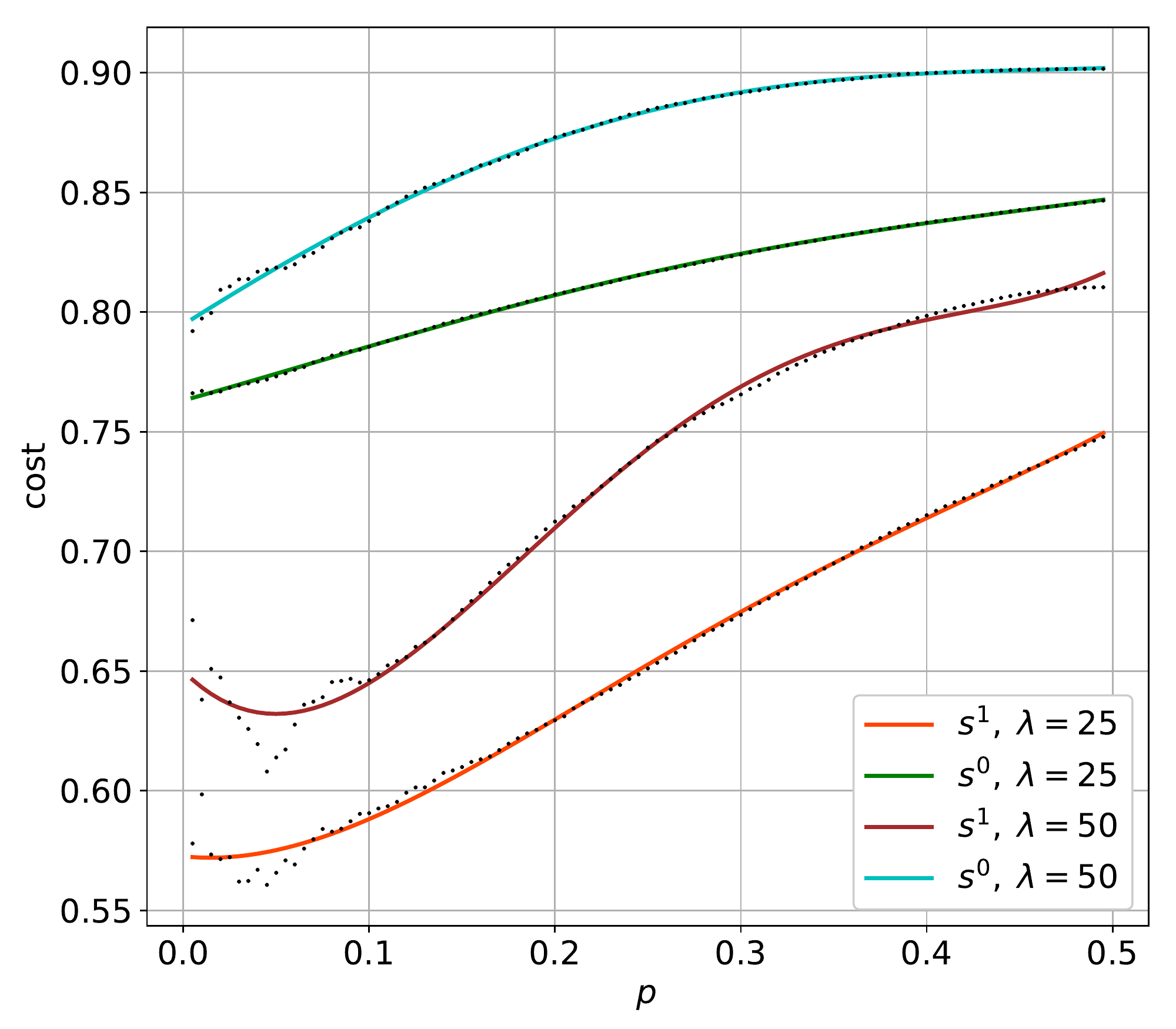} 
\includegraphics[width=7.4cm,height=6.6cm]{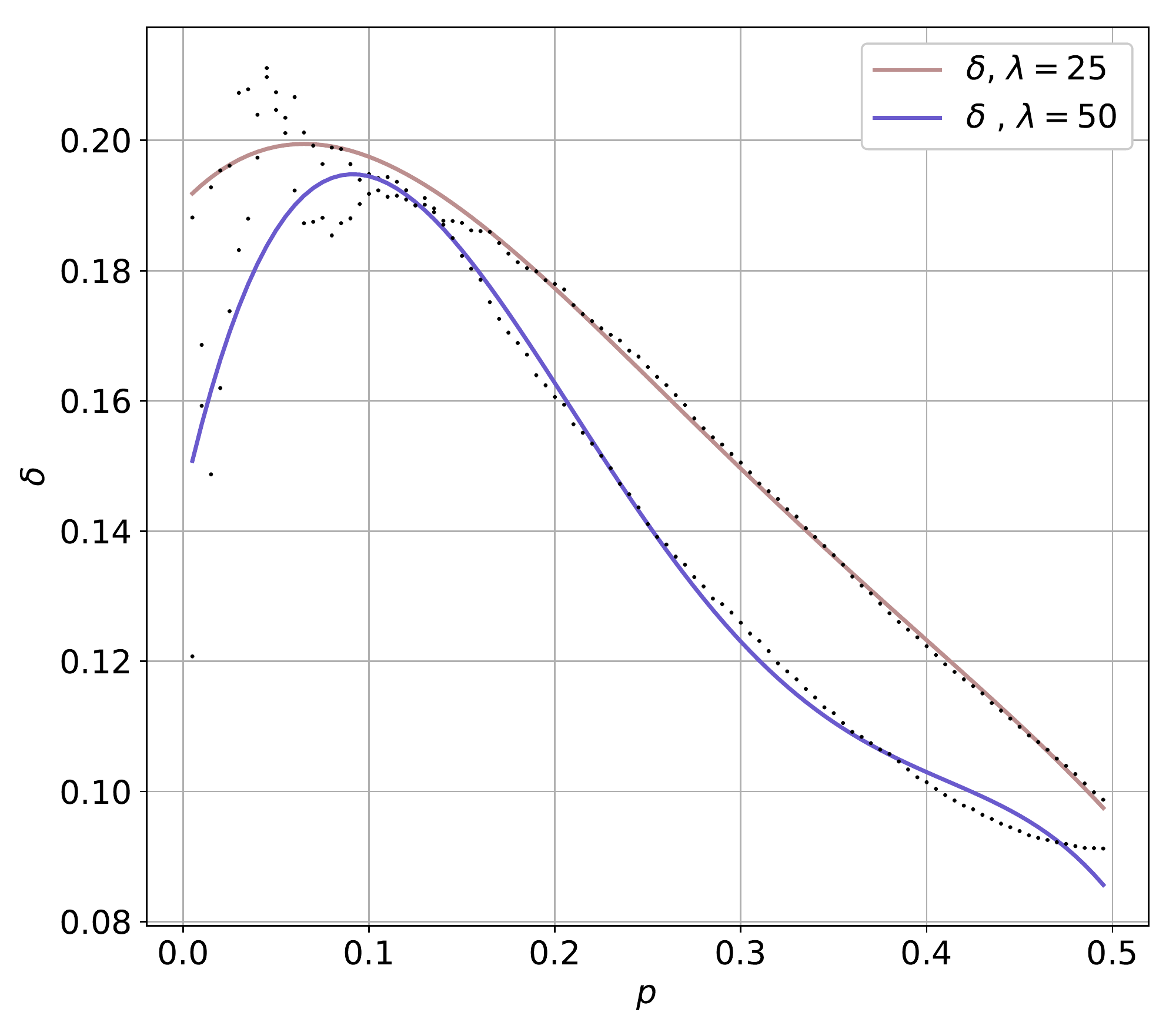}
\caption{Costs (a) and gain (b) of 
the strategy $\s^1$ over $\s^0$;  for $\alpha=0.5$.}
\label{a05}
\end{figure}
Figure~\ref{a05}(a) represents a typical graph of average costs 
of transactions under~$\s^0$ and transactions under~$\s^1$ as 
a function of fraction~$p$, for a higher~$\alpha$. 
In that case, even though the average costs of transactions 
under~$\s^0$ and transactions under~$\s^1$ do not coincide 
for any reasonable~$p$ (meaning that, here, the Nash equilibrium 
will be met at~$\theta=1$), the typical difference between 
the possible pure strategies (that, from now on, we will 
call absolute gains) will be low, as depicted on Figure~\ref{a05}(b).

\begin{figure}
\centering
\includegraphics[width=12cm,height=7cm]{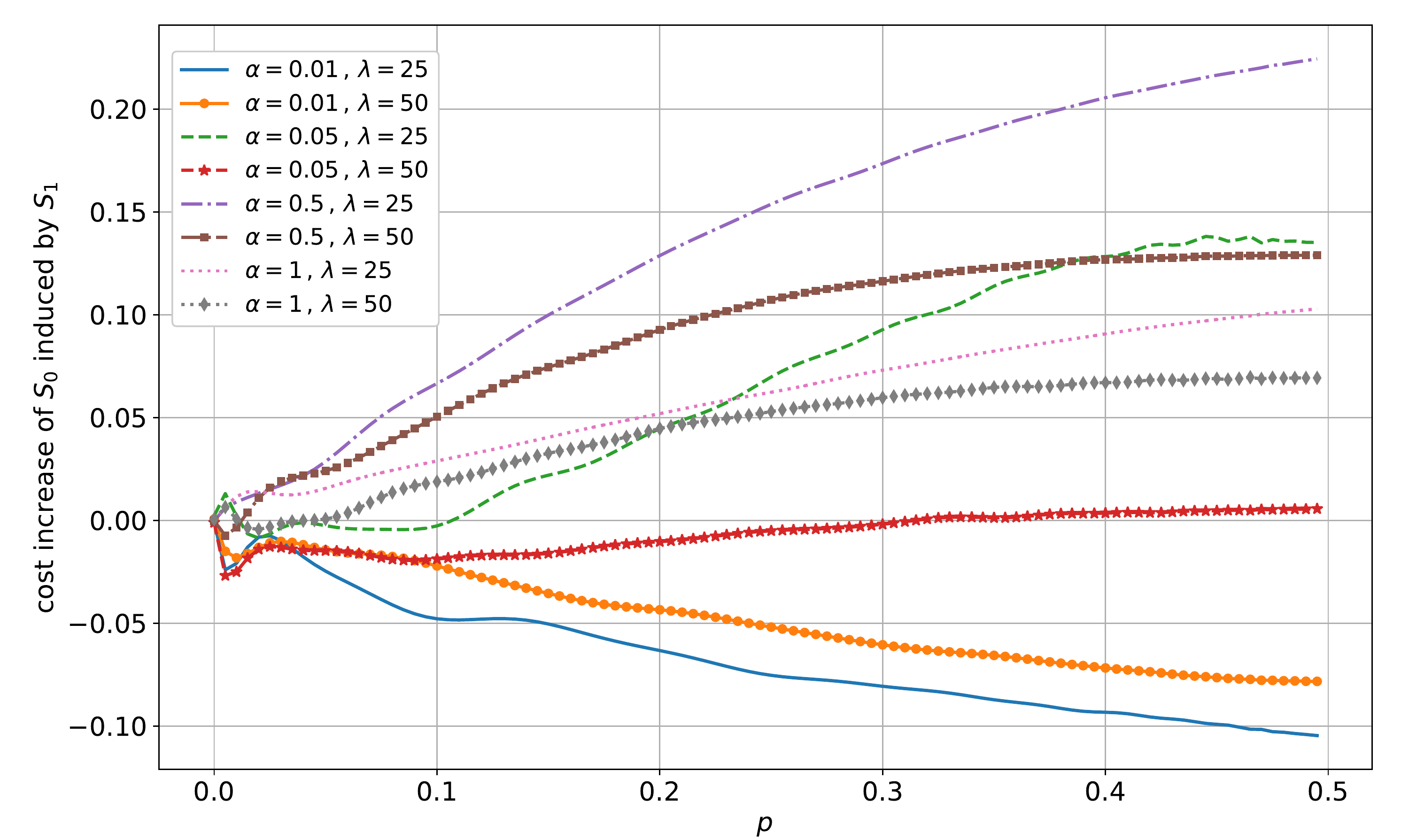}
\caption{Relative cost increase of the transactions issued by 
the strategy~$\s^0$
 induced by the presence of transactions emitted by 
the strategy~$\s^1$.}
\label{cost_increase}
\end{figure}
Figure~\ref{cost_increase} shows the average cost increase
 imposed on the nodes following the default strategy by the nodes
 issuing transactions under~$\s^1$. Let~$W(p)$ be the non-greedy 
nodes costs depicted in Figure~\ref{a05}(a). The cost increase is
 calculated as $(W(p)-W(0))/W(0)$, so it will be the relative
 difference of the cost of a non-selfish node in the presence of 
a fraction $p$ of selfish transactions and the cost of 
a non-selfish node when there are no selfish transactions at all. 
This difference is low, meaning that the presence of selfish nodes 
do not harm the efficiency of the non-selfish nodes. 
Note that this difference is small for all reasonable values of~$p$,
 but even for the larger simulated values of $p$, the difference is still less 
than~25\%. An interesting phenomenon, as shown in the same graph, 
is that the average cost increase imposed on the non-greedy nodes 
may actually be negative. For low values of $\alpha$, just a small
 fraction of the transactions under~$\s^0$ will share the approved 
tips with the transactions under~$\s^1$. 
This fraction of transactions will approve overcrowded tips, 
and will have their costs increased. All the other transactions 
 under~$\s^0$ will have their sites less crowded, since an increase 
in~$\s^1$ will mean a decrease in competition over these transactions.
 Finally, on average, the honest nodes will have their costs decreased.

\begin{figure}[!ht]
\hspace{4cm} (a) \hspace{7.4cm} (b)\\
\includegraphics[width=7.4cm,height=6.6cm]{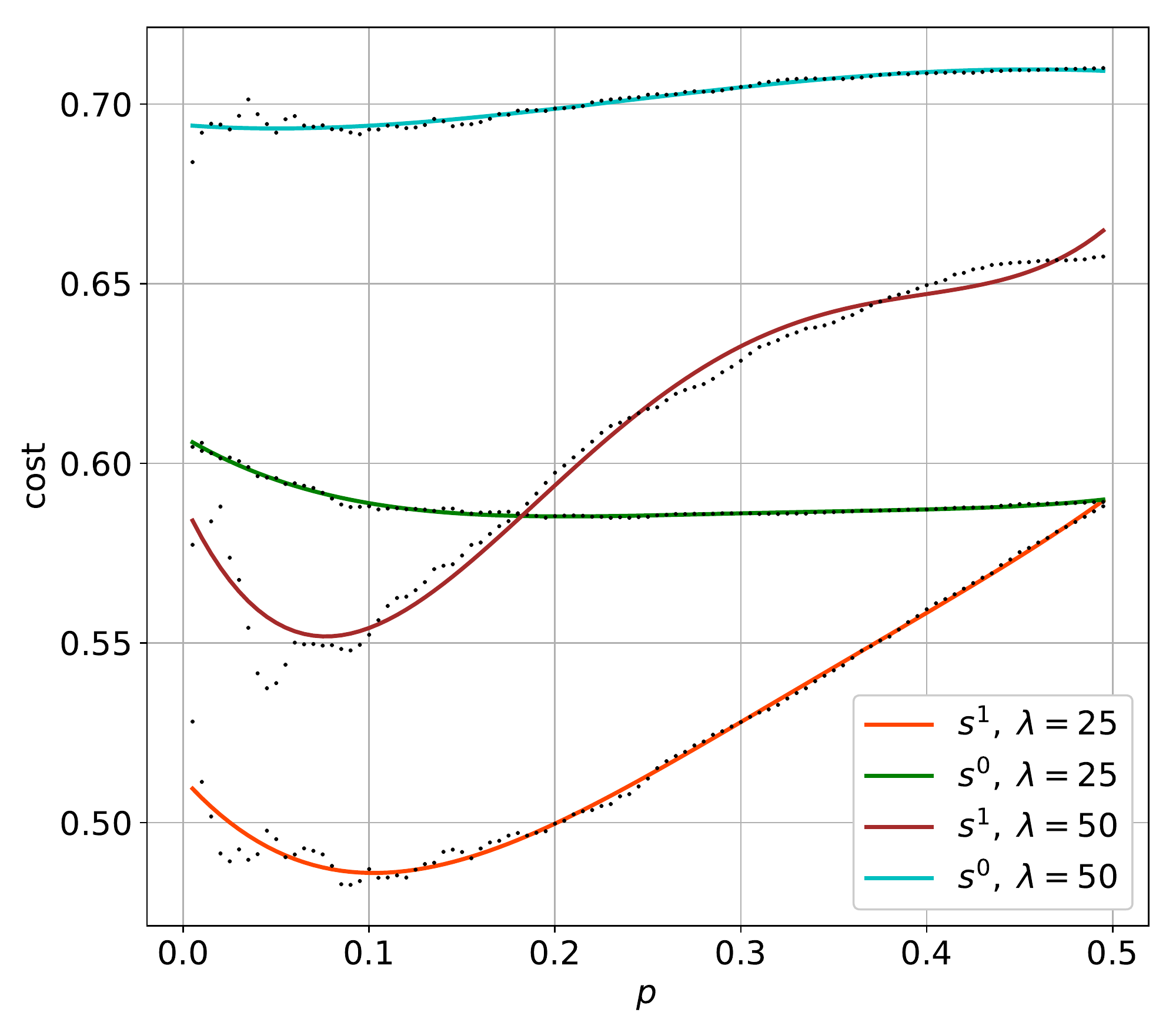} 
\includegraphics[width=7.4cm,height=6.6cm]{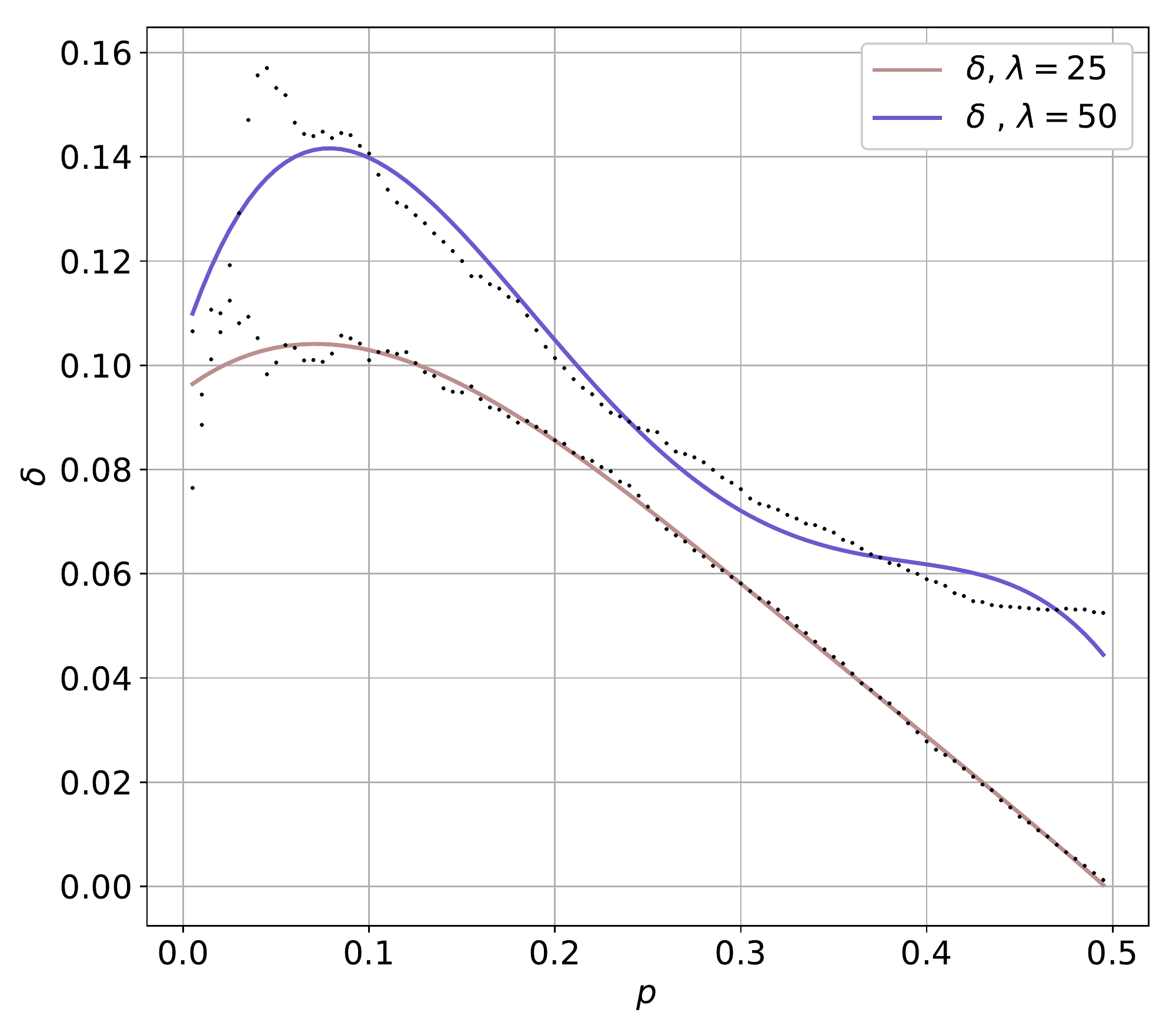}
\caption{Costs (a) and gain (b) of the strategy $\s^1$ 
over $\s^0$; for $\alpha=0.05$.}
\label{a005}
\end{figure}

\begin{figure}[!ht]
\hspace{4cm} (a) \hspace{7.4cm} (b)\\
\includegraphics[width=7.4cm,height=6.6cm]{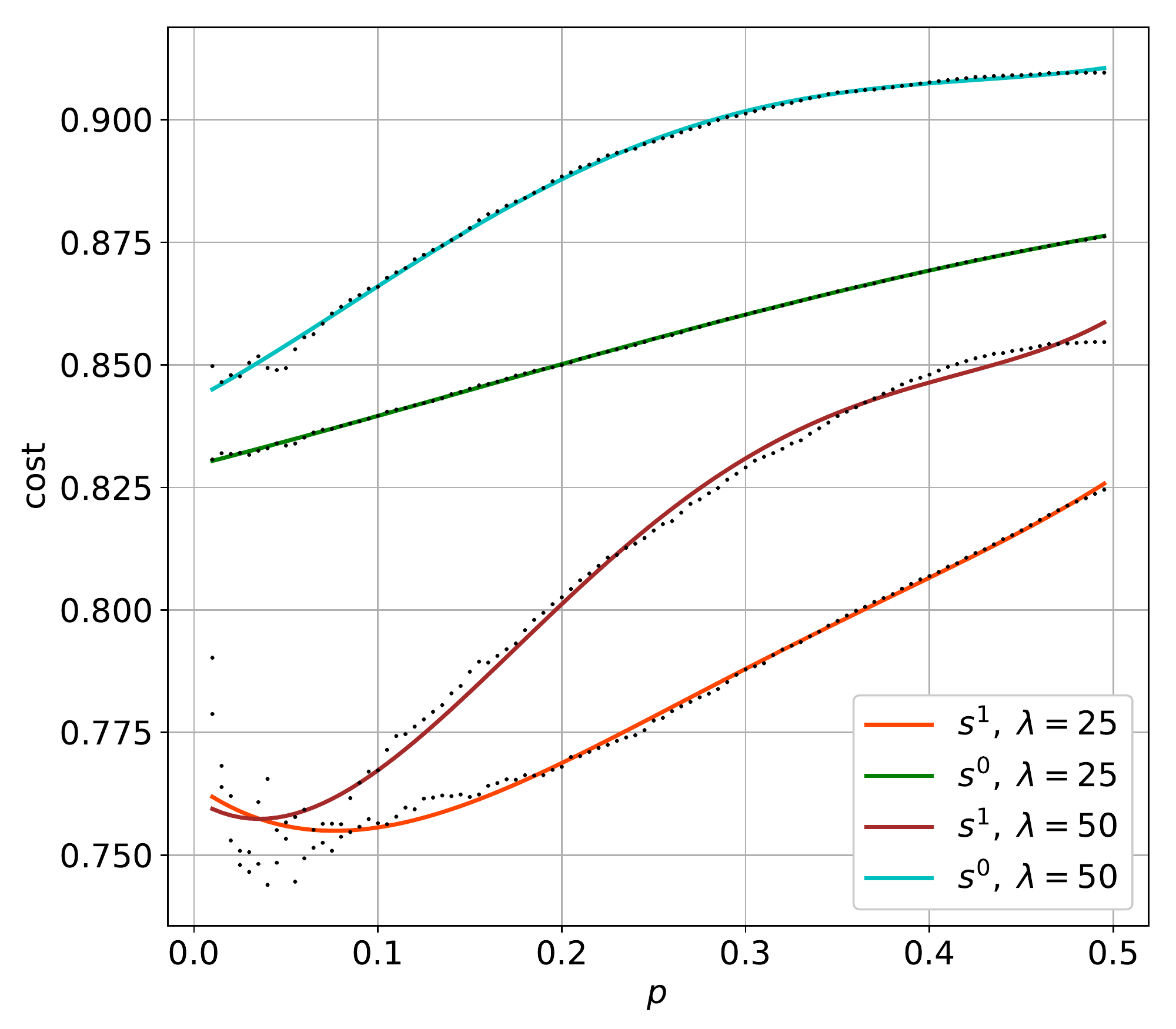} 
\includegraphics[width=7.4cm,height=6.6cm]{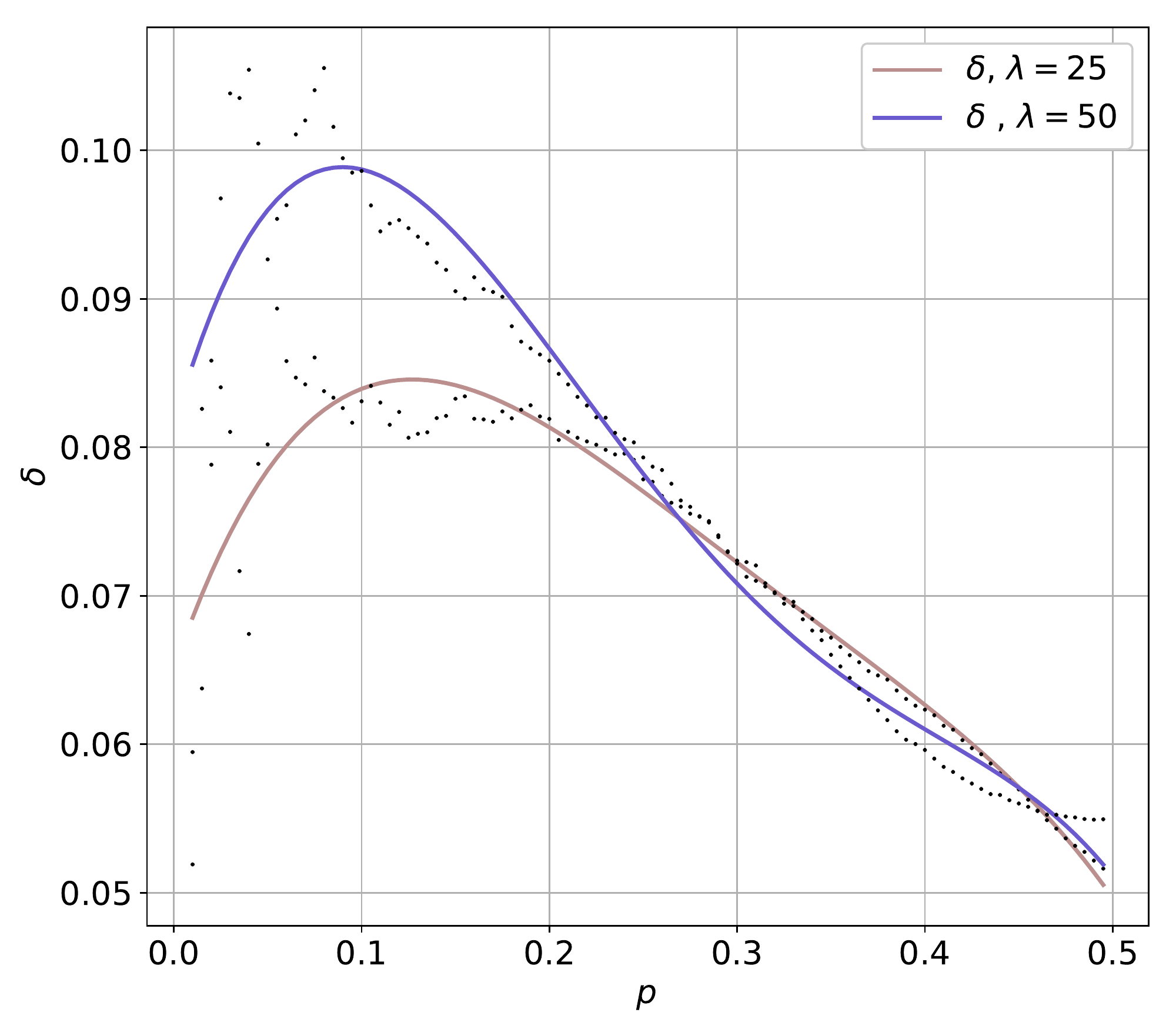}
\caption{Costs (a) and gain (b) of the strategy $\s^1$ 
over $\s^0$; for $\alpha=1$.}\label{a1}
\end{figure}

Figures~\ref{a005} and~\ref{a1} are analogous to the first figures,
 for other values of~$\alpha$ and~$\lambda h$; 
part~(a) of each figure represents average costs and 
part~(b) absolute gains.




\section{Conclusions and future work}
\label{s_conclusions}
 In the first part of this paper, we prove the existence of
 (``almost symmetric'') Nash equilibria for a game in 
the tangle where a part of players tries to optimise their 
attachment strategies. In the second part of the paper, 
we numerically determine, for a simple space strategy 
and some range of parameters, where these equilibria are located.
 
Our results show that the studied selfish strategy outperform 
 the non-selfish ones by a reasonable order of magnitude. 
The data show a 25\% (in the most extreme scenario) difference 
in the nodes gains, which in some situations, may be large enough.
 Nevertheless, the computational cost of a selfish strategy is
 intrinsically larger than the computational cost of the 
non-selfish strategies, since the selfish strategy uses 
the probability distribution of the tips, which is costly 
to calculate for a random walk with backtracking. They will also
 have to monitor the tangle, to know its parameters 
(like~$\lambda$, $h$ etc) and act accordingly. 
Also, even a extreme scenario, where almost half of
 the transactions were issued by a selfish node, 
is not enough to harm the non-selfish ones in a meaningful way. 

 On the other hand, our results raise further questions. 
The obtained data exhibit a deep qualitative dependence 
on the parameter~$\alpha$ of the simulation. 
This parameter is related to the randomness of the random walk:
a low $\alpha$ implies a high randomness; a higher~$\alpha$ 
implies a low randomness, meaning that the walk will be almost
 deterministic. Further simulations will be done in order to 
study the effect of that variable in the equilibria. 
 Also, we only studied equilibria for a given cost, relative to 
the probability of confirmation of the transactions in a certain
 interval of time. Since this probability depends heavily 
on the interval of time chosen (because the probability distribution
 of the confirmations is far from uniform), another time intervals, 
that will have another practical meaning, must be analysed.
 
 Finally, the equilibrium in the multidimensional strategy space 
should be studied in a more quantitative and analytic way, since 
it should depend strongly on~$\alpha$ and~$p$; and until now it
 was studied in just a narrow range of parameters. 
Further research will also be done in order to optimise the default 
tip selection strategy in a way that minimises this cost imposed by
 the selfish strategies. Through implementing research methods
 and techniques from the cross-reactive fields of measure theory, 
game theory, and graph theory, progress towards resolving 
the tangle-related open problems has been well under way and 
will continue to be under investigation.

As already mentioned, in this paper we consider 
only ``selfish'' players, i.e., those who only care about their own 
costs but still want to use the network 
in a legitimate way. We do not 
 consider at all the case when there are ``malicious'' 
ones, i.e., those who want to disrupt the network even
 at a cost to themselves.
We are going to treat several types of attacks against the 
network in the subsequent papers. 
Some preview of this ongoing work is 
available in~\cite{P18}.

\section*{Acknowledgements}
The authors thank Alon Gal, Gur Huberman, Bartosz Kumierz, 
John Licciardello, Andreas Penzkofer, Samuel Reid,
 and Clara Shikhelman for valuable comments and suggestions.
The authors are also grateful to the anonymous reviewers 
 for carefully reading the first version of this paper
and providing valuable comments and suggestions.

\end{document}